\documentclass[12pt]{amsart}

\usepackage{amsmath}
\usepackage{amssymb}
\usepackage{array}
\usepackage{enumitem}
\usepackage{hyperref}
\usepackage{verbatim} 
\usepackage{color}

\theoremstyle{plain}
\newtheorem{theorem}{Theorem}[section]
\newtheorem{lemma}[theorem]{Lemma}

\newtheorem{corollary}[theorem]{Corollary}

\newtheorem{definition}[theorem]{Definition}
\theoremstyle{remark}
\newtheorem{remark}{Remark}[section]
\newtheorem{example}{Example}[section]
\newtheorem*{notation}{Notation}
\newtheorem*{acknowledgment}{Acknowledgment}

\numberwithin{equation}{section}

\newcommand{\bA}{\mathbb{A}}
\newcommand{\mbA}{{\mathring \bA}}
\newcommand{\mbB}{{\mathring \bB}}

\newcommand{\bB}{\mathbb{B}}
\newcommand{\bD}{\mathbb{D}}
\newcommand{\B}{\mathbb{B}}

\newcommand{\K}{\mathbb{K}}
\newcommand{\bK}{\mathbb{K}}

\newcommand{\R}{\mathbb{R}}
\newcommand{\Z}{\mathbb{Z}}

\newcommand{\PP}{\mathbb{P}}

\newcommand{\bI}{\mathbb{I}}

\newcommand{\g}{\mathfrak{g}}

\newcommand{\XX}{\mathfrak{X}}

\newcommand{\Gl}{\mathrm{Gl}}
\newcommand{\GL}{\mathrm{GL}}
\newcommand{\SL}{\mathrm{SL}}

\newcommand{\OO}{\mathrm{O}}

\newcommand{\Hom}{\mathrm{Hom}}

\newcommand{\Aut}{\mathrm{Aut}}

\newcommand{\End}{\mathrm{End}}

\newcommand{\Gras}{\mathrm{Gras}}

\newcommand{\ad}{\mathrm{ad}}

\newcommand{\id}{\mathrm{id}}

\newcommand{\Ad}{\mathrm{Ad}}

\newcommand{\set}{\mathrm{set}}
\newcommand{\Alg}{\mathrm{Alg}}
\newcommand{\Walg}{\mathrm{Walg}}
\newcommand{\Weil}{\mathrm{Wsp}}
\newcommand{\Wman}{\mathrm{Wman}}
\newcommand{\Balg}{\mathrm{Balg}}
\newcommand{\Mor}{\mathrm{Mor}}

\newcommand{\Der}{{\mathrm{Der}}}
\newcommand{\Infaut}{{\mathrm{Infaut}}}

\newcommand{\eps}{\varepsilon}

\newcommand{\inv}{^{-1}}

\newcommand{\ssk}{\smallskip}
\newcommand{\nin}{\noindent}

\newcommand{\ul}{\underline}

\begin{document}

\title[Weil Spaces]{Weil Spaces and Weil-Lie groups}

\author{Wolfgang Bertram}

\address{Institut \'{E}lie Cartan Nancy \\
Lorraine Universit\'{e}, CNRS, INRIA \\
Boulevard des Aiguillettes, B.P. 239 \\
F-54506 Vand\oe{}uvre-l\`{e}s-Nancy, France}

\email{\url{wolfgang.bertram@univ-lorraine.fr}}



\begin{abstract} 
We define {\em Weil spaces, Weil manifolds, Weil varieties} and
{\em Weil Lie groups} over an arbitrary commutative base ring $\K$
(in particular, over discrete rings such as $\K = \Z$), and we develop the basic theory of such spaces,
leading up the definition of a Lie algebra attached to a Weil Lie group.
By definition, the category of Weil spaces is the category of functors from $\K$-Weil algebras to sets;
thus our notion of Weil space is similar to, but weaker than the one of  {\em Weil topos} defined by E.\ Dubuc
(\cite{Du79}). In view of recent result on Weil functors for manifolds over general topological 
base fields or rings (\cite{BeS}), this generality is the suitable context to formulate and prove
general results of infinitesimal differential geometry, as started by the approach developed in \cite{Be08}.
\end{abstract}

\maketitle

\section{Introduction}

The present work is a contribution to general {\em infinitesimal Lie theory}, and to the general 
theory of {\em infinitesimal spaces}.  In preceding work \cite{Be08}, based on \cite{BGN04},
we have been able to describe basic features of Lie theory for {\em Lie groups and manifolds 
over general topological base fields or rings}.  This approach is very satisfying in many 
respects, but still has the drawback that it relies on topology, hence does not apply to spaces defined
over discrete rings such as $\Z$; related to this, it does not lead to cartesian closed categories, nor
does it apply to singular spaces.  In the present work, we 
 introduce the category
of {\em Weil spaces and Weil laws}, and the one of {\em Weil Lie groups}, 
which do not have these drawbacks. 
 
 \ssk
The basic ideas how to achieve this goal can be traced back to Weil's paper \cite{We53}, which influenced both the development
of  algebraic geometry and of  {\em synthetic differential geometry}.  
As Weil puts it (loc.\ cit.), 
``Cette th\'eorie ... a pour but de fournir, pour le calcul diff\'erentiel d'ordre infinit\'esimal quelquonque sur une
vari\'et\'e, des moyens de calcul et des notations intrins\`eques qui soient aussi bien adapt\'es \`a leur object,
et si possible, plus commodes que ceux du calcul tensoriel classique pour le premier ordre.''\footnote{``The aim of the theory is to furnish, for infinitesimal differential calculus of arbitrary order on a manifold,
the tools of calculus, and intrinsic notions that are as suitable and, possibly, more flexible than usual first order tensor calculus.''} 
 We show that this aim can be reached in a  simple and general way in the framework of the {\em category
 of functors from $\K$-Weil algebras to sets}. The approach may 
appear primitive,  but
it seems to us that it is supported by recent developments, such as general infinite dimensional Lie theory.
I will, first of all, describe the approach, and then discsuss  the relation with
other  work.

\subsection{From polynomial laws to Weil laws}

Already undergraduate students are taught to distinguish between a
{\em polynomial} $P \in \K[X]$ and the {\em polynomial map} $\K\to \K$
 induced by $P$. For several, possibly infinitely many variables, 
N.\ Roby  (\cite{Ro63}) explains that {\em polynomial laws} play the same
r\^ole with respect to
 {\em polynomial maps} between
general modules over a ring $\K$ 
  (for a summary, see the appendix of \cite{Lo75}): a ``polynomial'' $\ul P$
  between $\K$-modules $V$ and $W$ is something that can be ``extended'' to a map $P^\bA$
  between $V^\bA = V\otimes_\K \bA$ and $W^\bA$, for any ring extension $\bA$ of the
  base ring $\K$. 
It is common for such concepts that the ``underlying'' polynomial map $P = P^\K$
need not determine the abstract object ``polynomial'', but in certain contexts (e.g., infinite fields) it does.

\ssk
For smooth maps $f$, at a first glance, there is no such thing as a ``scalar extension'' $f^\bA$.
However, a smooth map $f$ always admits an extension by its {\em tangent map} $Tf$, and we have shown in \cite{Be08} that
$Tf$ can rightly by interpreted as a {\em scalar extension of $f$ by the ring $T\K=\K \oplus \eps \K$
($\eps^2=0$) of dual numbers}. 
This generalizes  for all ring extensions of $\K$ by {\em algebras of infinitesimals}, nowadays called
{\em Weil algebras}:
these are commutative unital $\K$-algebras of the form
\[
\bA = \K 1 \oplus \mbA,
\]
where $\mbA$ is a {\em nilpotent ideal}, moreover free and finite-dimensional over $\K$. 
For usual, real manifolds, the theory of {\em Weil functors} (see \cite{KMS}) shows that a smooth map
$f:M \to N$ admits an extension to a map $f^\bA:M^\bA \to N^\bA$, for every Weil algebra
$\bA$. This extension   behaves like a {\em tangent map of kind $\bA$};
therefore the notation
$T^\bA f: T^\bA M \to T^\bA N$ is  often used in the literature.
Indeed, when $\bA = T\K$ (dual numbers), then
$TM:=T^{T\K} M$ is precisely  the ``usual'' (first) tangent bundle, when $\bA = TT\K$, then
$T^\bA M = TTM$, is the second order tangent bundle, and so on --
this has  been generalized, and at the same time conceptually explained, in \cite{Be08, So12, BeS}:
most importantly, the map $f^\bA := T^\bA f$ is indeed {\em smooth over $\bA$}, which fully justifies
to consider it as an  ``$\bA$-scalar extension of $f$''.
Motivated by this, we define {\em Weil spaces and Weil laws} following the same pattern 
as in  Roby's definition of polynomial laws (section \ref{sec:Weilsp}):
a {\em Weil space} is a functor $\ul M$ from the category $\ul{\Walg}_\K$ of $\K$-Weil algebras to the category of
sets;  it  assigns to every $\K$-Weil algebra $\bA$
a set $M^\bA$ that plays the r\^ole of an $\bA$-tangent bundle over the base
$M = M^\K$.  {\em Weil laws} $\ul f$ are the natural morphisms, i.e., {\em natural transformations},
between two such functors: for each Weil algebra $\bA$, there is an ``$\bA$-tangent map''
$f^\bA:M^\bA \to N^\bA$, depending functorially both on $f$ and on $\bA$. 
As explained above, usual smooth manifolds and maps furnish
an example of this pattern.

\subsection{Weil manifolds}
Every $\K$-module $V$ defines a ``flat'' Weil space, that can be used as model space
to define manifolds in the usual way via ``gluing data'': for each Weil algebra $\bA$,
$V^\bA$ can be described as
$$
V^\bA = V \otimes_\K  \bA = V \otimes_\K (\K \oplus \mbA) = V \oplus (V \otimes_\K \mbA)
= V \times  V_\mbA .
$$
In the same way  we can
define $U^\bA := U \times (V\otimes_\K \mbA)$ for any non-empty subset $U$ of $V$; this presentation
makes explicit the fibered structure of $U^\bA$ over $U = U^\K$.
Since the purely set-theoretic ``gluing data'' of a manifold (by forgetting about topology)
can be formulated in terms of subsets of $V$ and bijections between such subsets, 
a {\em Weil manifold} can be defined to be a functor from $\ul{\Walg}_\K$ into the
category of ``set theoretic manifolds'' (section \ref{sec:Weilmf}).

\ssk
The absence of a ``usual'' topology gives rise to some uncommon features; but this is remedied 
by always speaking of manifolds {\em with atlas}: the atlas is part of the 
structure of a manifold, and since it is given by a covering, it defines a topology which we call the
{\em atlas-topology}.
Structures such as dimension and the Lie algebra depend on this atlas --
just as  for usual Lie groups they depend on topology:  e.g., 
a  usual Lie group $G$  with discrete topology gives rise to a zero-dimensional Lie
group $G^{discr}$, and  $\R^n$ gives rise to $n+1$ different Lie groups
$\R^k \times (\R^{n-k})^{discr}$. 

\subsection{Weil varieties}
An important feature of Weil manifolds, compared to general Weil spaces, is that they
are ``infinitesimally linear'' (in synthetic differential geometry one uses the term
{\em microlinear}): 
the tangent spaces (fibers of $TM$ over $M$) are {\em linear spaces}, as they should,
and hence the fibers of the double tangent bundle $TTM$ are {\em bilinear spaces}
(see \cite{Be08}), and so on: the geometry of higher order tangent objects is polynomial. 
Since  Weil manifolds are not the only Weil spaces having such properties (e.g.,
affine algebraic varieties share them),  it is useful to define  the category
of {\em Weil varieties} as the infinitesimally linear Weil spaces (section \ref{sec:Weilvar}).
In order to understand Lie brackets and other bilinear objects, we prove some basic 
facts on {\em second order tangent bundles}.

\subsection{Weil Lie groups}
Clearly, usual Lie groups are generalized by group objects in the three categories
discussed so far (Weil spaces, Weil varieties, Weil manifolds).
It turns out that, for developing basic Lie theory, the category of Weil varieties is best
adapted, hence we define a {\em Weil Lie group} to be a group object in this category. 
This is the good setting to  define a bilinear {\em Lie bracket}
and to generalize the theory of the {\em higher order tangent groups $T^k G$ and $J^k G$}
from \cite{Be08} (section \ref{sec:WeilLie}). 
In principle, we follow here the presentation given by Demazure and Gabriel in 
\cite{DG}, II, \S 4, just by forgetting about the language of schemes and sheaves and instead insisting on
functorial properties of Weil algebras.
In the same way, one may recast the infinitesimal  theory of {\em symmetric spaces}, as developed in \cite{Be08}.

\subsection{Symmetric  spaces and differential geometry}
In order to keep this work concise, we do not develop general notions of differential geometry
in any detail;  we just give one major result (theorems \ref{th:Aut1}, \ref{th:Aut2}) illustrating
 that the conceptual framework of Weil varieties is
well-adapted to generalize, and to  simplify by the way, most  of the concepts and results
from  \cite{Be08}.
In particular, as explained in loc.\ cit.,  we consider the {\em theory of connections} to be
the core of infinitesimal geometry. The thesis \cite{So12} contains a good deal of the
theory to be developed in this context, and these topics certainly deserve further study.

\subsection{Discussion}
As said above, our notion of Weil space is, in some sense, quite ``primitive'': we have an
 imbedding of the category of usual manifolds into the category of Weil spaces, but this imbedding
is {\em not full}, that is, for two classical manifolds $M$ and $N$, not every morphism $\ul f$
between the corresponding 
Weil spaces $\ul M$ and $\ul N$ is induced by a usual   smooth map $f:M \to N$.
Indeed, our category of Weil spaces is essentially what E.\ Dubuc in \cite{Du79} calls (in the real
case) the {\em Weil topos}, and Dubuc explicitly describes examples of morphisms
in the Weil topos that are not induced by smooth maps (loc.\ cit., p.\  258/59).
The reason for this failure is  that morphisms in the Weil topos only have
``infinitesimal regularity'', but there is no reason why infinitesimal regularity should imply
local regularity, such as continuity or usual smoothness.  In order to get full imbeddings, Dubuc constructs in loc.\ cit.\ more
sophisticated topoi, which since then have been further 
refined and used  in synthetic differential geometry
(cf.\ \cite{Ko06, La87, MR91}).

\ssk
However, I believe that there are several good reasons to develop this ``primitive'' approach, in spite
of the apparent ``lack of fullness'': 

\ssk
(1) Fullness is not needed for defining and studying infinitesimal objects such as Lie brackets, connections,
 curvature tensors and the like. The methods are simple, they apply directly  to usual manifolds, even in ``very
 infinite dimensional settings'', and hence give conceptual and general proofs for results in this framework. 
 
\ssk
(2) As Ivan Kol\'a\v r stresses in several of his papers (see \cite{Kolar, Kolar2}), there  are two approaches to 
Weil functors which are sort of dual to each other, and which he calls {\em covariant}, resp.\  {\em contravariant}.
As far as I understand, the topoi mentioned above, in the spirit of algebraic geometry, follow the
contravariant approach: inevitably, at some place,
the relation of a vector space (or $\K$-module) $V$ with its dual space $V^*=\Hom_\K(V,\K)$
 comes into the game, be it in the language of function algebras, or sheaves, or schemes.  
This makes things technically more complicated, and, at some point, breaks down: it works well
in finite-dimensional, or other sufficiently regular situations, but becomes problematic already for real topological vector spaces 
 beyond Fr\'echet spaces. This is even more so for
general modules $V$ over rings (not assumed to be free), where $V^*$ may be
very poor:   if there are not enough linear functionals, then there won't be enough smooth scalar functions
or germs neither!
On the other hand,  the differential calculs developed in \cite{BGN04},
and the subsequent work \cite{Be08, Be13, BeS, So12}, are   ``purely covariant'':
at no point, the theory relies on function algebras or dual spaces (e.g., vector fields are not {\em defined}
as derivations, although they may of course {\em act} as such).
 This is the reason why this approach works so well in arbitrary dimension. 

\ssk
(3) Finally, I conjecture  that it is possible to obtain fullness also by a ``purely covariant approach'', which then
would combine advantages of the sophisticated topoi with those of a simple approach. 
Moreover, understanding how this works should also lead to  important insights into the structure
of differential calculus itself: 
the key observation (see \cite{Be08b, Be13})
 is that every usual smooth map admits even more general scalar extensions than those by
Weil algebras; thus there should be a class of bundle algebras 
(cf.\ Section \ref{sec:Walg}), strictly bigger than the one of Weil algebras, such that the functor category
from this category to \ul{set}  really corresponds to a ``smooth''  category.
We hope to be able to develop this approach in subsequent work.

\ssk
Related to the preceding item, another, very interesting, topic for further research is to extend the theory to
classes of {\em non-commutative} (bundle) algebras, and foremost, to {\em super-commutative
algebras} -- see recent work \cite{AHW, AlL, BCF}   (where {\em super Weil algebras} are introduced) 
making already important steps into  this direction.

\begin{acknowledgment}
I thank Arnaud Souvay for many stimulating discussions on topics related to this work.
\end{acknowledgment}

\begin{notation}
$\bK$ is  a unital commutative ring that may be considered to be fixed once and for all (think of $\K=\R$),
$\bA$, $\bB$,... are associative commutative and unital $\K$-algebras.
Categories and functors are usually underlined; 
the category of $\K$-algebras is denoted by $\ul{\Alg}_\K$,
and the category of sets and maps by \ul{set}; categories whose objects are functors
(functor categories) are doubly underlined, e.g., $\ul{\ul{\Weil}}_\K$.
\end{notation}

\section{Bundle algebras, Weil algebras, and vector algebras}
 \label{sec:Walg}

We define some categories of algebras that play an important r\^ole in differential geometry,
and we develop the basic theory as far as needed for our purposes.

\begin{definition}\label{def:Walg}

\begin{enumerate}
\item
A {\em scalar extension} of $\K$ is an associative commutative
unital $\K$-algebra $\bA$, or, in other words, a morphism
$\phi : \K \to \bA$ of unital commutative rings. Scalar extensions of $\K$
 form a category $\ul\Alg_\K$,
morphisms being $\K$-algebra homomorphisms.
\item
 A {\em $\K$-bundle algebra}  is an associative  unital $\K$-algebra of the form 
$$
\bA = \K \oplus \mbA
$$
where $\K =\K 1$ and $\mbA$ is an ideal of $\bA$, called the {\em fiber of $\bA$},
 which is assumed to be free and finite-dimensional as
a $\bK$-module. 
{\em Morphisms of bundle algebras} are  algebra homomorphisms preserving 
fibers. 
$\K$-bundle algebras form a category denoted by $\ul\Balg_\K$.
\item
 A {\em $\K$-Weil algebra}  is a commutative bundle algebra such that the fiber $\mbA$
  is a {\em nilpotent}
 ideal, 
 called the {\em ideal of infinitesimals}. $\K$-Weil algebras
 form a category denoted by \ul{$\Walg_\K$}, 
 morphisms being  bundle algebra morphisms. 
 \item
 The {\em height} of a Weil algebra $\bA$ is the smallest natural number $k$ such
 that $\mbA^{k+1} = 0$.  Weil algebras of height $\leq k$  form a category denoted by
 $\ul\Walg_{\K}^k$.
 \item
A Weil algebra $\bA$  is called a {\em vector algebra} if
it is of height one, i.e.,  if
the ideal of infinitesimals has zero product:
 $\forall a,b \in \mbA$: $ab = 0$.
 \end{enumerate} 
\nin
\end{definition}

\nin
Elements in a bundle algebra will be written
 $(x,a) \in \K \oplus \mbA$, so  the  product is 
\begin{equation}\label{eqn:B-notation}
(x,a) \cdot (x',a') = (xx', xa' + ax' + aa'),
\end{equation}
and in a vector algebra we have, moreover, $aa'=0$.

\begin{definition}
For any bundle algebra $\bA = \K \oplus \mbA$, the projection
$\pi:\bA \to \K$ is a morphism with kernel $\mbA$, called the {\em projection onto the base ring},
and the natural injection $\zeta : \K \to \bA$ is a morphism called the {\em zero section}.
Thus $\bA$ is a ring extension of $\K$, via the zero section, but $\K$ is also a
ring ``extension'' of $\bA$, via the projection. 
However, we will rather say that $\K$ is obtained from $\bA$ by {\em scalar restriction}. 
\end{definition}

\begin{example} 
Our main examples of bundle and Weil algebras are:

\begin{enumerate}
\item
The {\em tangent algebra of $\K$}, or {\em dual numbers over $\K$}, is the vector 
algebra
\begin{equation*}
T\K := \K [\eps] := \K[X] / (X^2) = \K \oplus \eps \K, \qquad \eps^2 = 0 .
\end{equation*}
\item
The {\em idempotent algebra of $\K$} is the bundle algebra
\begin{equation*}
I \K := \K [j] := \K[X] / (X^2-X) = \K \oplus j  \K, \qquad j^2 = j .
\end{equation*}
\item
The {\em $k$-jet algebra} is the Weil algebra of height $k$
\begin{equation*}
J^k \K := \K[X] / (X^{k+1}) = \K \oplus (\delta \K \oplus \ldots \oplus \delta^k \K),
\qquad \delta^{k+1}=0 .
\end{equation*}
\item
Let $I_0:=(X)$ be the ideal of polynomials in $\K[X]$ that vanish at $0$.
For $(s_1,\ldots,s_k) \in \K^k$ we define a bundle algebra
$$
\qquad 
\bA = \K[X]/ (X (X-s_1) \cdots (X - s_k))  = \K \oplus I_0 /   (X (X-s_1) \cdots (X - s_k)) .
$$
This  bundle algebra is used in {\em simplicial differential calculus}, see  \cite{Be13}. 
\item 
More generally, let $I_0 = (X_1,\ldots,X_m)$ be the ideal of polynomials 
in $m$ variables that vanish at the origin of $\K^m$, and
$J \subset I_0$ an ideal having a free  complement. Then
$\bA = \K[X_1,\ldots,X_m]/ J$ is a bundle algebra, and it is a Weil algebra if 
$I_0^r \subset J$ for some $r$, i.e., if it is a quotient of the 
{\em Weil algebra of $m$-dimensional $r$-velocities}
\[
{\mathbb D}_m^r := \K[X_1,\ldots,X_m]/ I_0^r .
\]
Every Weil algebra can be presented in this way;  such presentations play an important 
r\^ole in synthetic differential geometry.
\item
{\em Tensor products} and {\em fiber products} can be used to construct new bundle or Weil algebras
from given ones, see below. Of particular importance will be the {\em second tangent algebra}
\[
\qquad \quad
TT\K := T\K \otimes_\K T\K = (\K \oplus \eps_1 \K) \otimes (\K \oplus \eps_2 \K) =
\K \oplus (\eps_1 \K \oplus \eps_2 \K \oplus \eps_1 \eps_2 \K)
\]
with $\eps_1^2 = 0 = \eps_2^2$. It is isomorphic to $\K[X_1,X_2] / (X_1^2,X_2^2)$.
\end{enumerate} 
\end{example}

\begin{definition}
In the categories of bundle or Weil algebras, {\em pushouts} and {\em pullbacks} 
are given by the following constructions: 
given two algebras in these categories, $\bA = \K \oplus \mbA$ and $\bB = \K \oplus \mbB$, 
\begin{enumerate}
\item
the pushout (with respect to the zero sections) is the {\em tensor product over $\K$}, which 
can be identified with the $\K$-module
$$
\bA \otimes_\K \bB =  (\K \oplus \mbA) \otimes_\K (\K \oplus \mbB) =
\K \oplus (\mbA \oplus \mbB \oplus \mbA \otimes_\K \mbB),
$$
with product given by:

$ (x; a,b,u \otimes v) \cdot (x';a',b',u' \otimes v') = \bigl(xx';  xa'+x'a+aa', $
$$
 \qquad \qquad xb'+x'b+bb', xu' \otimes v' + x' u\otimes v + 
 uu'\otimes vv' + a \otimes b' + a' \otimes b \bigr),
$$
\item
the pullback (with respect to the projections)
is the {\em fibered product} or {\em Whitney sum over $\K$}, which can be identified 
with the $\K$-module
$$
\bA \times_\K \bB = \K \oplus (\mbA \oplus \mbB) 
$$
with product given by:
$$
(x; a,b) \cdot (x';a',b') = 
(xx'; xa'+x'a+aa',xb'+x'b+bb') .
$$
\end{enumerate}
\end{definition}
%
\begin{remark}
The height of $\bA \times_\K \bB$ is the maximum of the heights of $\bA$ and $\bB$,
whereas the height of $\bA \otimes_\K \bB$ is their sum.
In particular, the Whitney sum of two vector algebras is again a vector algebra, whereas
their tensor product is of height two (cf.\ the example of $TT\K$ given above). 
\end{remark}

Obviously, there is a natural sequence of morphisms of Weil (or bundle) algebras
\begin{equation}\label{eqn:WeilSeq}
\begin{matrix} 
\K & \to & (\mbA \otimes_\K \mbB ) \oplus \K  &
\to & \bA \otimes \bB & \buildrel{p_{\bA,\bB}} \over \to &
\bA \times_\K \bB & \to & \K
\end{matrix}
\end{equation}
which is {\em exact} in the following  sense: 

\begin{definition}\label{def:ideal}
An {\em ideal} of a bundle (resp.\  Weil) algebra $\bA  = \K \oplus \mbA$ is an ideal
$\bI$ of $\bA$ contained in $\mbA$ that is free as $\K$-module and admits some free 
$\K$-module complement $C$.
Note that $\bA /  \bI$ then is again a bundle (or Weil) algebra, and so is
$$
\hat \bI := \K \oplus \bI .
$$
We then  say that the following is a {\em short exact sequence of Weil or bundle algebras}:
$$
\begin{matrix} 0 & \to & \hat \bI & \to & \bA & \to & \bA / \bI & \to & 0 \end{matrix} .
$$
The sequence (\ref{eqn:WeilSeq}) is exact, and the  {\em vertical bundle algebra of $\bA \otimes \bB$} is
its kernel: 
$$
\bA \odot_\K  \bB := \bA \odot \bB  := 
\widehat{\mbA \otimes_\K \mbB} =   (\mbA \otimes_\K \mbB ) \oplus \K ,
$$
Finally,
a  {\em vector ideal} is an ideal $\bI$ acting as zero on $\mbA$:
$\forall a \in \mbA$, $\forall i \in \bI$: $ai=0$.
\end{definition}

The  preceding constructions lead to  natural morphisms between bundle or
Weil algebras, all of which will have important  global counterparts:

\begin{lemma}\label{la:sum}\label{affine-lemma}
Let $\bA$ be a commutative bundle algebra. 
\begin{enumerate}
\item
The product map in $\bA$ gives rise to a morphism of bundle algebras 
$$
\mu : \bA \otimes_\bK  \bA 
 \to \bA, \quad (x,a) \otimes (y,b)  \mapsto (x,a)(y,b)=(xy,xb+ay+ab)  \, .
$$
\item
The {\em flip} or {\em exchange map} is an automorphism of bundle algebras:
$$
\tau : \bA \otimes_\bK \bA \to \bA \otimes_\bK \bA, \quad a \otimes b \mapsto b \otimes a .
$$
\item
 If $\bA$ is a vector algebra, then addition in $\mbA$ gives rise to a morphism
$$
\alpha  : \bA \times_\bK \bA \to \bA, \quad (x,a,b) \mapsto (x,a+b) \, ,
$$
and each scalar $r \in \K$ gives rise to an algebra endomorphism
$$
\rho_r: \bA \to \bA, \quad (x,a) \mapsto (x,ra) .
$$
\item
Let $\bI$ be a vector ideal in the bundle algebra $\bA$. Then the map 
$$
\beta : \hat \bI \times_\K \bA = \K \oplus \bI \oplus \mbA   \to \bA, \quad (x; i , a) \mapsto (x; i+a)
$$
is a bundle algebra morphism lying over $\bA / \bI$ in the sense that
$$
 \hat \bI \times_\K (\bA/\bI) = \K \oplus \bI \oplus (\mbA/\bI)   \to \bA/\bI , \quad (x; i , [a]) \mapsto (x; [i+a])
$$
is simply projection onto the second factor. 
\end{enumerate}
\end{lemma}

\begin{proof}
(1) and (2)   hold for any commutative associative algebra $\bA$. To prove (3),
 assume $\bA$ is a vector algebra. Then
\begin{eqnarray*}
\alpha ((x,a,b)\cdot (x',a',b'))&=& \alpha (xx', xa'+x'a,xb'+x'b) =
(xx',  xa'+x'a + xb'+x'b)  \cr
&=&
(x,a+b) (x',a'+b') = \alpha(x,a,b) \cdot \alpha(x',a',b')
\end{eqnarray*}
and $\rho_r ((x,a)(x',a')) = \rho_r (xx',xa'+x'a) = (xx', rxa' + rx'a)= \rho_r(x,a) \cdot \rho_r(x',a')$.

(4) Similarly, by direct computation,  with $ai' = 0 =a'i = ii'$,  
\begin{eqnarray*}
\beta ((x,i,a)\cdot (x',i',a'))&=& \beta (xx', xi'+x'i,xa'+x'a + aa') \cr
&=&
(xx',  xi'+x'i + xa'+x'a + aa')  \cr
&=&  (xx', x(a'+i') + x'(a+i) + (a+i)(a'+i')) \cr
&=& 
(x, i +a) (x',i'+a') = \beta(x,i,a) \cdot \beta(x',i',a') ,
\end{eqnarray*}
and the last claim is immediate since $[i+a]=[a]$.
\end{proof}

\section{Weil spaces and Weil laws}\label{sec:Weilsp}

\subsection{The category of Weil spaces}

\begin{definition}\label{def:Weisp1}
A {\em $\K$-Weil space}  is a covariant functor
$$
\ul{M} : \ul\Walg_\K \to \ul\set , \quad \bA \mapsto M^\bA  , \phi \mapsto M^\phi .
$$
A {\em morphism
 between $\K$-Weil spaces $\ul M$ and $\ul N$}, also called {\em $\K$-Weil law},
  is a natural transformation 
$$
\ul f : \ul M \to \ul N .
$$
In other words, for each $\bA \in \ul\Walg_\K$, there are mappings
$$
f^\bA :M^\bA \to N^\bA
$$  
varying functorially with $\bA$:  for any Weil algebra morphism $\phi:\bA \to \bB$ and Weil spaces
$\ul M, \ul N$, there are
maps $M^\phi$, $N^\phi$ such that the following diagram commutes:
\[
\begin{matrix} 
M^\bA & \buildrel{f^\bA}\over{\longrightarrow} & N^\bA \cr
M^\phi \downarrow \phantom{M_f}  & & \phantom{M_f} \downarrow N^\phi \cr
M^\bB & \buildrel{f^\bB}\over{\longrightarrow} & N^\bB 
\end{matrix}
\]
Weil spaces and Weil laws over $\K$ become a category, denoted by $\ul{\ul \Weil}_\K$, if we define the
{\em composition} $\ul g \circ \ul f$ of two $\bK$-Weil  laws $\ul f:\ul M \to \ul N$, $g:\ul N \to \ul P$ by
$$
\forall \bA \in \Walg_\K: \qquad (\ul g\circ \ul f)^\bA := g^\bA \circ f^\bA : M^\bA \to P^\bA.
$$
 Isomorphisms and automorphisms in the category of Weil spaces will be called {\em Weil isomorphisms}, resp.\ {\em Weil automorphisms}.
 \end{definition}

\begin{definition}\label{def:Weilsp2}
Replacing the category $\ul\Walg_\K$ by some other category of $\K$-algebras, we may
define categories of ``stronger'' or ``weaker'' Weil spaces.
For instance, taking  $\ul\Walg_\K^k$ (Weil algebras of height at most $k$),
we get a weaker category of ``$k$ times differentiable Weil spaces'', and
taking the full category of (free and finite dimensional) commutative $\K$-algebras, we get 
``algebraic Weil spaces''.\footnote{R.\ Lavendhomme
(\cite{La87}, p.\ 195, referring to Kock \cite{Ko77}) considers this category as a very simple, but logically satisfying model for the
axiomatic setting of synthetic differential gometry.}
\end{definition}

\nin
As said in the introduction, there should be some category of bundle algebras such that the ``stronger'' category
of Weil spaces defined by it should correpond to what one might call ``smooth Weil spaces''.

\begin{definition}  The {\em direct product} $\ul M \times \ul N$ of two Weil spaces is defined by
$$
\forall \bA \in \ul\Walg_\K: \qquad  (M\times N)^\bA := M^\bA \times N^\bA ,
$$
and the {\em direct product of Weil laws} $\ul f \times \ul g$ is defined similarly.
\end{definition}

\begin{definition}\label{def:subspace}
A {\em subspace} $\ul U$ of a Weil space $\ul M$ is a Weil space $\ul U$ such that, for
each Weil algebra $\bA$, $U^\bA$ is a subset of $M^\bA$.
If each $U^\bA$ is of cardinality $1$, then $\ul U$ is called a {\em point in $\ul M$}.
\end{definition}

\nin Obviously,
intersection $\ul U \cap \ul U'$ and union $\ul U \cup \ul U'$ of subspaces of $\ul M$, defined by
$$
(\ul U \cap \ul U')^\bA := U^\bA \cap (U')^\bA, \qquad
(\ul U \cup \ul U')^\bA := U^\bA \cup (U')^\bA
$$
are again subspaces of $\ul M$, and this holds also for arbitrary intersections and unions.

\subsection{Fibered structure over the base $M$}

\begin{definition}
The {\em  underlying set} of a $\K$-Weil space $\ul M$ is the set
 $M:= M^\K$, and the {\em underlying map} of a $\K$-Weil law $\ul f : \ul M \to \ul N$ is the
 map $f:=f^\K: M \to N$.
 \end{definition}
 
The underlying  map $f$ does in general not determine the abstract law $\ul f$
(see examples, next section). 
Applying functoriality of $\ul f$ to projection
$\pi:\bA \to \K$, resp.\ zero section $\zeta:\K\to \bA$,
 of a Weil algebra $\bA$,
we get:

\begin{lemma}
Weil spaces and Weil laws are {\em fibered over their underlying set theoretic objects} in the sense that,
for all Weil algebras $\bA$, the following diagram commutes
\[
\begin{matrix} 
M^\bA & \buildrel{f^\bA}\over{\longrightarrow} & N^\bA \cr
M^\pi \downarrow \phantom{M_f}  & & \phantom{M_f} \downarrow N^\pi \cr
M  & \buildrel{f }\over{\longrightarrow} & N
\end{matrix} \, .
\]
Moreover, these fibratations  have canonical {\em zero sections} in the sense that
the following diagram commutes
\[
\begin{matrix} 
M^\bA & \buildrel{f^\bA}\over{\longrightarrow} & N^\bA \cr
M^\zeta \uparrow \phantom{M_f}  & & \phantom{M_f} \uparrow N^\zeta \cr
M  & \buildrel{f }\over{\longrightarrow} & N
\end{matrix} \, .
\]
\end{lemma}

\nin The term ``fibration'' is used here in the abstract set-theoretic sense; there is no
condition of ``local triviality''  (since so far we do not consider any topology). 

\begin{definition}
The bundle $T^\bA M := M^\bA$ over $M$ will be called the {\em $\bA$-tangent bundle of $M$}, and
$T^\bA f := f^\bA$ the {\em $\bA$-tangent map of $f$}.
The fiber of $M^\bA$ over $x \in M$ will often be denoted by 
$$
T_x^\bA M:= (M^\bA)_x,
$$
and the map between the fibers over $x$ and $f(x)$ by
$$
T_x^\bA f : T_x^\bA M \to T_{f(x)}^\bA N .
$$
These maps preserve origins in fibers; we write $0_x$ or $\zeta(x)$ for the origin in $T_x^\bA M$.
When $\bA = T\K$, the ring of dual numbers over $\K$, then we simply write $TM$ and $Tf$
and just call them {\em tangent bundle}, resp.\ {\em tangent map of $f$}.
\end{definition}

In Chapter \ref{sec:Weilvar}
 we will introduce sufficient conditions ensuring that the fibers of $TM$ are linear spaces and
that tangent maps are linear maps in fibers, and we will see that this condition is satisfied by $TM$
for our main examples.

\subsection{Transitivity of extension}
For any smooth manifold $M$ over a topological ring $\K$, the extension $T^\bA M$ is
as manifold smooth over the ring $\bA = T^\bA \K$ (see \cite{BeS}) -- this is a conceptual version of what
Weil in \cite{We53} calls ``transitivit\'e des prolongements'' (loc.\ cit., th\'eor\`eme 5). 
In the present context, it  reads as follows:

\begin{theorem}[Transitivity of extension]\label{th:transitivity}
Given two Weil algebras $\bA$ and $\bB$ with a morphism $\bB \to \bA$ (in other words,  $\bA$ is a $\bB$-algebra), 
there is a natural {\em functor of scalar extension of
Weil spaces from $\bB$ to $\bA$}
$$
T^{\bA,\bB}:  \ul{\ul\Weil}_\B \to \ul{\ul\Weil}_\bA
$$
associating to a $\bB$-Weil space $\ul M$ the $\bA$-Weil space given on the level of objects by
$$
\ul \Walg_\bB \to \ul\set, \qquad 
\bD \mapsto M^{\bD \otimes_\bB \bA} 
$$
and similarly on the level of morphisms. 
In particular, there are natural {\em functors of scalar extension and scalar restriction}
\begin{align*}
T^{\bA,\K} : & \ul{\ul\Weil}_\K \to \ul{\ul\Weil}_\bA, \cr
T^{\K,\bB} : &
\ul{\ul\Weil}_\bB \to \ul{\ul\Weil}_\K, 
\end{align*}
and the composed ``$\bA$-tangent functor'' is defined by
$$
T^\bA : = T^{\K,\bA} \circ T^{\bA,\K} : 
 \ul{\ul\Weil}_\K \to \ul{\ul\Weil}_\K, \quad \ul M \mapsto \ul{M^\bA}
$$
All of  these functors are {\em product preserving functors} in the sense of \cite{KMS}, i.e.,  they are
 compatible with direct products.
\end{theorem}

\begin{proof}
This follows directly from the corresponding transitivity properties of the scalar extension functor
(tensor product) on the level of rings, in particular from the natural isomorpism
$\bD \otimes_\bB (\bB \otimes_\K \bA) = \bD \otimes_\K \bA$.
(Note that only properties of bundle algebras are needed for this; nilpotency of the ideals plays
no r\^ole here.)
\end{proof}

\nin
The notation $\ul M^{\bA,\bB}:= T^{\bA,\bB}\ul M$ may also be useful: it indicates that we look at
``the $\bA$-bundle $M^\bA$, seen over $M^\bB$''.  The letter $T$ now stands for 
 a covariant bifunctor 
$$
T : \ul\Walg_\K \times \ul{\ul\Weil_\K} \to \ul{\ul\Weil_\K}, \quad (\bA, \ul M) \mapsto \ul{M^\bA} .
$$

\subsection{Internal hom-spaces, and cartesian closedness}
It  is well-known that  spaces of smooth maps between manifolds are rarely manifolds.
One of the motivations to develop topos theory and synthetic differential geometry is to
define categories which behave better with respect to this issue.
As said in the introduction, our model is rather primitive, compared to topoi used in
synthetic differential geometry; thus   the proof of the following result is
standard in category theory. 

\begin{theorem}[Internal hom-spaces, and cartesian closedness]
Let $\ul M$ and $\ul N$ be $\K$-Weil spaces. Then the morphisms ($\K$-Weil laws) from $\ul M$ to $\ul N$,
$$
\ul X := \ul \Mor (\ul M,\ul N),
$$
form again a $\K$-Weil space, by considering, for each $\bA \in \Walg_\K$, the set
$$
 X^\bA :=  \{  f^\bA : M^\bA \to N^\bA \mid \, \ul f \in \ul X \} 
$$
with natural transformations given for a morphism $\phi:\bA \to \bB$ by
$$
X_\phi : (X^\bA \to X^\bB, \, f^\bA \mapsto (u \mapsto f^{\bB} (M^\phi (u))).
$$
It follows that  the category $\ul{\ul \Weil}_\K$  of $\K$-Weil spaces  is
{\em cartesian closed}.
\end{theorem}

\begin{proof}
Direct check of definitions, see \cite{BW}, Section 2.1, Theorem 4.
\end{proof}

\section{Examples of Weil spaces}\label{sec:Weilex}

\subsection{$\K$-modules}
Any $\K$-module $V$ gives rise to a Weil space
\begin{align*}
\ul {V}:  
\ul{\Walg}_\K & \to \ul\set,  \cr
 \bA &  \mapsto V^\bA := V \otimes_\K  \bA , \quad 
(\phi : \bA \to \bB) \mapsto (\id_V  \otimes_\K \phi : V^\bA \to V^\bB).
\end{align*}
In fact, it gives rise to the stronger structure of an ``algebraic Weil space'' (cf.\
Definition \ref{def:Weilsp2}), since it defines also a functor on {\em all} scalar extensions
$$
\ul {\bf V}:  \ul{\Alg}_\K \to \ul \set,   \qquad  \bA \mapsto V^\bA := V \otimes_\K  \bA , 
\phi \mapsto \id_V \otimes_\K  \phi .
$$
The morphisms in this stronger sense are precisely the {\em polynomial laws}:

\begin{definition}[N.\ Roby, \cite{Ro63}]
 A {\em polynomial law} between two $\K$-modules $V$ and $W$  is a 
natural transformation $\ul f: \ul{\bf V} \to \ul{\bf W}$.
\end{definition}

As shown in loc.\ cit.\ (cf.\ also \cite{Lo75}, Appendix),
 the underlying map $f: V \to W$ of a polynomial law corresponds
to classical concepts of polynomial mappings. 
By restriction to the subcategory $\ul\Walg_\K$, any polynomial law gives rise to a Weil 
law; in particular, linear maps define Weil laws. 
When nothing else is
said, a $\K$-module $V$ will always be considered as Weil space $\ul V$ in the way defined above.

\begin{definition} \label{def:triv}
Let $V$ be a  $\K$-module.
The {\em canonical trivialization}  of the Weil space $\ul V$ is given by
the decomposition, for each Weil algebra $\bA$,
$$
V^\bA = V \otimes (\bK \oplus \mbA) =  V \oplus (V \otimes \mbA) =:  V \oplus V^\mbA .
$$ 
\end{definition}

Using the trivialization,  tangent maps give  rise to {\em differentials}:
elements of $V^\bA$ are written
in the form
$u = (x,v)$ or $u=x \oplus v$. A Weil law $\ul f:\ul V \to \ul W$ is written, with respect to the
trivializations,  $f^\bA(x,v)= (g(x,v),h(x,v))$. Since $f^\bA$ is fibered over $f:V\to W$, the first 
component is just $f(x)$, so that, letting $d^\bA f(x)v:=h(x,v)$, 
\begin{equation}\label{eqn:Adiff}
f^\bA (x,v) = \bigl( f(x), d^\bA (x)v) \bigr) = f(x) \oplus d^\bA f(x) v.
\end{equation}
The condition $z^W \circ f = f^\bA \circ z^V$ gives us
$d^\bA f (x)0 = 0$.

\begin{example} If $\bA = T\K =\K \oplus \eps \K$ ($\eps^2 = 0$), then
$V^\bA = TV = V \oplus \eps V$, and, letting $df(x):= d^{T\K} f(x)$,
 (\ref{eqn:Adiff}) is the analog of the usual
first order Taylor expansion
$$
Tf (x+\eps v) = f(x) + \eps df(x) v .
$$
In section \ref{sec:Weilvar} we show that, for all $x\in V$, the map
$df(x):V \to W$ is linear.
\end{example}

\begin{example} Assume $f:V \to W$ is $\K$-linear. Then
$Tf:TV \to TW$ is its $T\K$-linear extension, whence 
$Tf(x+\eps v) = f(x) + \eps Tf(v) = f(x) + \eps f(v)$.
Comparing with the preceding formula, we get 
$df(x)v=f(v)$;
thus $df(x)=f$ is constant.
Likewise, the differential of a bilinear map $b:V \times V \to W$ is computed ``as usual''.
\end{example}

\begin{example}
For any Weil algebra $\bB$, $\ul \bB$ is the Weil space associating to $\bA$ the tensor product $\bB \otimes_\K \bA$.
In particular, the {\em affine line} $\ul \K$ associates to $\bA$ its underlying set.
\end{example}

\begin{example} Let  $\ul V=\ul \K$ be the affine line.
The product map $m:\K\times \K \to \K$ is bilinear, hence polynomial, and gives rise to a Weil
law $\ul m$, where  $m^\bA$ is simply the product map of $\bA$. Similarly for the addition map.
Summing up, the Weil functor $T^\bA$, applied to the base ring $\K$,
yields the ring structure of $\bA$.
Similarly, if $\g$ is a $\K$-algebra (associative, Lie, or other), then
$\ul \g$ is the law given by scalar extension $\bA \mapsto \g^\bA = \g \otimes_\K \bA$
(which is again associative, Lie, or other), and hence $\ul \g$ becomes an algebra object (associative, Lie or other)
in the category of Weil spaces. 
For instance, we may speak of the
associative Weil law $\ul{M(n,n;\K)}$.
\end{example}

\begin{example}
Vor $V = 0$,  we obtain
the terminal object associating $\K$ to each $\bA$ 
\begin{equation}\label{eqn:point}
\ul 0  : \, \bA \mapsto \K .
\end{equation}
A point $\ul p$  in $\ul M$ is the same as a morphism $\ul p:\ul 0 \to \ul M$. 
\end{example}

\subsection{Domains in $\K$-modules} 
Using the trivialization of $V^\bA$, we can mimick the definition of the ``usual'' tangent bundle
for any subset of $V$ (thought of as ``open''):

\begin{definition}
Let $V$ be a $\K$-module. A {\em domain in $\ul V$} is  the Weil space $\ul U$ given by a 
non-empty subset
$U \subset V$: 
to $\bA$, it assigns the subset of $V^\bA$ defined by
$$
 U^\bA := U \times (V \otimes_\bK \mbA) = U \times V^{\mbA} \subset
 V^\bA = V \times V^{\mbA}
$$
and to a morphism $\phi:\bA \to \bB$ the restriction of $\id \otimes \phi$ from $U^\bA$ to $U^\bB$. 
\end{definition}

\begin{example} We may take a singleton $U = \{ p \}$. Then the ``tangent space'' is $T_p \{ p \} = \eps V$, corresponding
to the idea that $U$ is considered as ``open'' in $V$. 
We denote by $\ul{0_V}$ the functor defined by the singleton $\{ 0_V \}$:
\begin{equation}
\ul{0_V} : \bA \mapsto \{ 0_V \} \times V^{\mbA} .
\end{equation}
Thus $\ul{0_V}$ represents the ``collection of all  $\bA$-tangent spaces of $V$ at $0$'', whereas $\ul 0$ represents 
the ``common origin'' in these tangent spaces. 
\end{example} 

\begin{example} 
(Inversion law.) 
Let $U:=\K^\times \subset V =\K$ and $\ul i : \ul U \to \ul U$ the law ``inversion''.
Since $\mbA$ is nilpotent, the set $U^\bA = \K^\times \times \mbA$ is precisely
the set $\bA^\times$ of invertible elements in $\bA$, and then we let
$i^\bA (u) = u\inv$, the inversion map of $\bA$.
This defines a Weil law. Note that, if $\K = \Z$, whence $U = \{ 1,-1 \}$, the base map
is $i^\K = \id_U$, but the Weil law $\ul i$ is different from the identity law $\ul\id_U$
(for instance, $i^{T\K}(1 + \eps v) = 1 - \eps v$;
only if $\K = \Z / 2 \Z$,  we cannot distiguish $\ul i$ and $\ul{\id}_U$).
\end{example}

\subsection{Subspaces defined by equations, and affine algebraic Weil spaces}\label{subsec:Zf}
Let $\ul M$ be a Weil space and
 $S$  some set of scalar valued Weil  laws
$\ul f: \ul{M} \to \ul \K$. Then we define its {\em zero locus} to be the functor
$$
\ul{Z(S)}: \bA \mapsto Z(S)^\bA:=\{ x \in M^\bA \mid \, \forall \ul f \in \ul S :
f^\bA (x) = 0 \} 
$$
and associating to $\phi:\bA \to \bB$ the restriction of $M^\phi$;  
then $\ul{Z(S)}$ is a $\K$-Weil space, in fact, a subspace of $\ul M$, called {\em 
the subspace defined by the equations $S$}.
In case $\ul M = \ul V$ is a $\K$-module and the equations are polynomial laws, we may
call it also an  {\em affine algebraic Weil space}. 
Of course, instead of scalar valued laws one might also take laws with other target spaces. 

\begin{example}
The Weil space $\ul{\SL(n,\K)}$ is of this type: it associates to $\bA$ the set
$\SL(n,\bA)$. Here, $S = \{ \ul{\det - 1} \}$. Since $\det(1 + \eps X) = 1 + \eps \, {\rm tr} ( X)$,
the tangent space at $1$ is the space of matrices of trace zero.
\end{example}

\subsection{Weil  manifolds}
\label{sec:manifolds}

A {\em Weil manifold} is a Weil space with the additional structure of being a 
{\em functor from the category of Weil algebras into the category of set theoretic manifolds modelled on
some $\K$-module $V$}:

\section{Weil manifolds}\label{sec:Weilmf}

\subsection{Set theoretic manifolds}
This is what remains if, in the usual definition of topological
manifolds, we retain the atlas and forget about the given topology:

\begin{definition}
A {\em set theoretic manifold over $\bK$} is given by 
$(M,V,(U_i,\phi_i,V_i)_{i\in I})$, where $M$ is a set, $V$ is a $\bK$-module,
and, for each $i$ belonging to an index set $I$,
$U_i \subset M$ and $V_i \subset V$ are non-empty subsets such that $M = \cup_{i\in I} U_i$,
and $\phi_i : U_i \to V_i$ is a set-theoretic isomorphism (bijection).
We then also say that ${\mathcal A} = ( U_i,  \phi_i,V_i)_{i\in I}$
is an {\em atlas on $M$ with model space $V$}, and we say that the topology
generated by the sets $(U_i)_{i\in I}$ on $M$ is the
{\em atlas-topology on $M$}. 
The atlas is called {\em saturated} if the $U_i$ form a basis of the atlas-topology.
\end{definition}

At this stage, the ring $\bK$ does not play any r\^ole.
If the atlas is made of ``big'' charts, then the atlas topology will not be separated (e.g.,
projective spaces $M=\K \PP^n$, when $\K$ is a field, with the usual atlas given
by $n+1$ canonical charts). On the other hand, the definition does not exclude charts
given by singletons, so that the atlas-topology may be discrete.  In practice, when
$M$ already carries some topology, one will require that the atlas-topology is coarser
than the given one. 
Given an atlas, we let for $(i,j) \in I^2$, 
\begin{equation}
U_{ij}:= U_i \cap U_j \subset M , \qquad V_{ij}:= \phi_j(U_{ij}) \subset V ,
\end{equation}
and the {\em transition maps} belonging to the atlas are defined by
\begin{equation}
 \phi_{ij}:= \phi_i\circ{\phi_j^{-1}} \vert_{ V_{ji} }:V_{ji}\to V_{ij} .
\end{equation}
They are bijections satisfying the {\em cocycle relations}
\begin{equation}\label{eqn:cocycle}
 \phi_{ii}=\id \quad  \mbox{and} \quad    \phi_{ij}\phi_{jk}= \phi_{ik} \ \mbox{ (where defined)}.
\end{equation}

\begin{definition}
A {\em morphism of set theoretic manifolds} $(M,{\mathcal A})$, $(N,{\mathcal A}')$
is an {\em atlas-continuous map} $f:M\to N$, i.e., a map which is continuous
with respect to the atlas-topologies on $M$ and $N$.
\end{definition}

The continuity condition  permits to recover the morphism from local data, see theorem below.
Obviously, set theoretic manifolds [modelled on $\K$-modules $V$] 
with their morphisms form  a category, denoted by
$\ul{\rm SetMf}_\K$. 
Recall  that a manifold can be reconstructed from local data, as follows:

\begin{theorem}[Reconstruction from local data]\label{th:reconstruct}
Assume  given the following data:
\begin{itemize}
\item a  $\K$-module $V$ (the model space), and an index set $I$,
\item  subsets  $ V_{ij} \subset V$,  for $i,j \in I$,
\item bijections  $( \phi_{ij} :V_{ij}\to V_{ji})_{i,j\in I}$ satisfying the cocycle relations
(\ref{eqn:cocycle}).
\end{itemize}
Then there exists a unique (up to isomorphism)
 set theoretic manifold $M$ having $V$ as model space and  the
$ \phi_{ij}$ as transition laws. 
Moreover, $f:M\to N$ is a morphism if and only if all 
$f_{ij}:= \psi_i \circ f \circ \phi_j\inv$ (restricted to suitable intersections of chart domains)
are morphisms.
\end{theorem}

\begin{proof}
Existence: 
 define $M$  to be the quotient  $M:=S/\sim$, where 
$S:=\{(i,x)|x\in V_{ii} \}\subset I\times V$ with respect to the equivalence relation
 $(i,x)\sim(j,y)$ if and only if $(\phi_{ij})(y)=x$. 
 We then put 
$V_i:=V_{ii}$, $U_i:=\{[(i,x)],x\in V_i\}\subset M$ and $\phi_i:U_i\to V_i, [(i,x)]\mapsto x$.
All properties, as well as uniqueness, 
are  now checked in a straightforward way; we omit the details (cf.\   \cite{BeS, Be13}.)
\end{proof}

\begin{example}
If all $U_{ij}$ are empty for $i \not= j$, then $M$ is just the disjoint union of the sets $V_i:=V_{ii}$.
On the other hand, every subset $U \subset V$ is a manifold (with $\vert I \vert = 1$).
\end{example}

\subsection{Weil manifolds}

\begin{definition}\label{def:Weil-manifold}
A {\em $\K$-Weil manifold (with atlas)} is a functor from the category of Weil algebras into the category
of set-theoretic manifolds; in other words, it
is a Weil space
$\ul M$ together with an {\em atlas}  $\ul{\mathcal A} = (\ul U_i, \ul \phi_i, \ul V_i)_{i\in I}$
modelled on some $\K$-module $V$, i.e., for each $\K$-Weil algebra $\bA$,
${\mathcal A}^\bA$ is an $\bA$-atlas on $M^\bA$ with model space $V^\bA$.
A {\em law of Weil manifolds} is a natural transformation
of Weil manifolds. 
\end{definition}

\nin
Weil manifolds with their laws obviously form a category, which we denote by
$\ul{\ul \Wman}_\K$.
For $i,i \in I$, we let $\ul V_{ij}:=\ul \phi_i ( \ul U_i \cap \ul U_j)$ (subspace of $\ul V_i$). 
Since the $\ul \phi_i$, are Weil isomorphisms, it follows that the {\em transition laws}
$$
\ul \phi_{ij}:=\ul \phi_i\circ{\ul \phi_j^{-1}} \vert_{\ul{\phi_j(V_{ji})}}:V_{ji}\to V_{ij}
$$
are Weil isomorphisms.  
From Theorem \ref{th:reconstruct} we get:

\begin{theorem}{\rm (Reconstruction from local data)}
Assume  given the following data:
\begin{itemize}
\item a  $\K$-module $V$ (the model space), and an index set $I$,
\item  Weil subspaces  $\ul V_{ij} \subset \ul V$,  for $i,j \in I$,
\item $\K$-Weil laws  $(\ul \phi_{ij} :\ul V_{ij}\to \ul V_{ji})_{i,j\in I}$ satisfying
for each Weil algebra $\bA$ the cocycle relations (\ref{eqn:cocycle}).
\end{itemize}
Then there exists a unique (up to isomorphism)
 Weil manifold $\ul M$ having $\ul V$ as model space and  the
$\ul \phi_{ij}$ as transition laws. 
\end{theorem}

\begin{example}
All usual (finite-dimensional real) manifolds are Weil manifolds over $\K=\R$  (\cite{KMS}, Theorem 35.13),
and all smooth manifolds over non-discrete topological fields or rings $\K$ are
$\K$-Weil manifolds (\cite{BeS}, Theorem 1.2).
\end{example}

\begin{example}
For any commutative ring $\K$, and any $\K$-module $V$ admitting a direct sum decomposition
$V = E \oplus F$,
the {\em Grassmannian of type $E$ and co-type $F$}, $M=\Gras_E^F(V)$ (space of all
$\K$-submodules isomorphic to $E$ and having a complement isomorphic to $F$) gives rise
to a Weil manifold: the Weil space structure is given by $M^\bA = \Gras_{E^\bA}^{F^\bA}(V^\bA)$,
and the charts by the natural affine space structure on the set of complements of some fixed module.
In particular, the {\em projective space laws}
$\ul{\K \PP^n} := \ul{\Gras_\K^\K (\K^{n+1})}$
are $\K$-Weil manifolds. 

\ssk
This example has a vast generalization: every {\em Jordan geometry} associated to
a {\em Jordan pair over $\K$} carries a structure of $\K$-Weil manifold, and so does
any {\em associative geometry} associated to an {\em associative pair over $\K$} (\cite{Be13b}). 
\end{example}

\section{Weil varieties}\label{sec:Weilvar}

\subsection{Infinitesimal linearity of Weil varieties}
The fact that Weil manifolds admit local trivializations, in the sense of definition \ref{def:triv}, implies rather directly
that they are {\em infinitesimally linear}: the fibers of the tangent bundle $TM$ are $\K$-modules.
For instance, tensoring
 the exact sequence (\ref{eqn:WeilSeq}):  $\bA \odot \bB  \to \bA \otimes \bB \to \bA \times_\K \bB$
for a pair $(\bA,\bB)$ 
of Weil algebras with a $\K$-module $V$ gives the decomposition
\begin{equation}
V^{\bA \otimes \bB} = V^{\mbA \otimes \mbB} \oplus V^\mbA \oplus V^\mbB \oplus V ,
\end{equation}
showing that the following sequence is {\em exact over the base $V$}, in the sense that 
it is an exact sequence in the fiber over $V$:
 \begin{equation}\label{eqn:WeilSeq'}
\begin{matrix} 
V & \to & V^{\mbA \otimes_\K \mbB } \oplus V  &
\to & V^{\bA \otimes \bB} & \buildrel{p^{\bA,\bB}} \over \to &
V^\mbA \oplus V^\mbB \oplus V  & \to & V
\end{matrix}
\end{equation}
This is already the main ingredient in the proof of

\begin{theorem}\label{th:mfd-var}
Assume $\ul M$ is a Weil manifold over $\K$.
\begin{enumerate}
\item
For all pairs $(\bA,\bB)$ of Weil algebras,
 there is a natural isomorphism between $M^{\bA \times_\K \bB}$ and
 the {\em fiber product} of $M^\bA$ and $M^\bB$ over the base $M$:
$$
\boxed{M^{\bA \times_\K \bB} = M^\bA \times_M M^\bB} .
$$
Naturality means here, that for any $\K$-Weil  law $\ul f : \ul M \to \ul N$,
$$
 f^{\bA \times_\K \bB} =
f^\bA \times_M f^\bB . 
$$
\item
An exact  sequence of Weil algebras
$\hat \bI \to \bA \to \bA / \bI$  (cf.\ definition \ref{def:ideal})  induces
an {\em exact sequence of bundles over $M$}
$$
\boxed{
\begin{matrix} M & \to & M^{\hat \bI} & \to & M^\bA & \to &  M^{\bA / \bI} & \to & M \end{matrix} }
$$
 i.e., in each fiber over $M$ the induced maps form  an
exact sequence of pointed sets: the inverse image of $0$ under  the second map is the image of the first.
\end{enumerate}
\end{theorem}

\begin{proof}
Assume first that  $\ul M = \ul V$
is given by a $\K$-module $V$.  Then, since tensor products over $\K$ are compatible with direct sum
decompositions over $\K$, we get from (\ref{eqn:WeilSeq}) the sequence (\ref{eqn:WeilSeq'}). 
Similarly, an exact sequence of Weil algebras as in def.\ \ref{def:ideal} (where
$\bA = \K \oplus \bI \oplus C$) gives rise by tensoring to
$$
V^\bA = V \otimes_\K (\K \oplus \bI \oplus C) =
V \oplus (V \otimes_\K \bI) \oplus (V \otimes_\K C) ,
$$
where $V \otimes_\K C = V_{\bA / I}$. 
From these identifications, both assertions (1) and (2) immediately follow.

If $\ul M$ is a general Weil manifold, then we apply the preceding arguments in each
domain $\ul V_{ij}$, and since the properties are of local (even: infinitesimal) nature, 
and chart changes induce isomorphisms on the set level, they hold also for $\ul M$. 
\end{proof}


\subsection{Weil varieties}
To formulate the general definition of Weil varieties,
we note that for any Weil space $\ul M$ and pair of Weil algebras $(\bA,\bB)$, there is a natural map
\begin{equation}
p^{\bA,\bB}: 
M^{\bA \times_\K \bB}\to M^\bA \times_M M^\bB 
\end{equation}
which, in the fiber over $x \in M$, for $u \in M^{\bA \times_\K \bB}$, is simply given by
\begin{equation}\label{eqn:FP2}
p^{\bA ,\bB} (u) = \bigl( p^\bA(u), p^\bB(u) \bigr) 
\end{equation}
where $p^\bA:M^{\bA \times_\K \bB} \to M^\bA$ 	and $p^\bB:M^{\bA \times_\K \bB} \to M^\bB$ are induced by
 $\bA \times_\K \bB \to \bA$ and $\bA \times_\K \bB \to \bB$.
This is a simple consequence of the very definition of the {\em fiber product of $M^\bA$ and $M^\bB$ over $M$}:
$
M^\bA \times_M M_{\bB} 
$
is defined to be
 the pushout in the category $\ul\set$, and it can be constructed as equalizer of the two projections
$\pi^\bA : M^\bA \to M$ and $\pi^\bB:M^\bB \to M$: 
\begin{equation}
M^\bA \times_M M_{\bB}  = \{ (u,v) \in M^\bA \times M^\bB \mid \, \pi^\bA (u) = \pi^\bB(v) \} ,
\end{equation}
and the fiber over $x \in M$ then is the direct product $T_x^\bA M \times T_x^\bB M$.
The map $p^{\bA,\bB}$ then exists by the universal property, and is explicitly given by (\ref{eqn:FP2}).

\begin{definition}
 A {\em Weil variety} is  a Weil space $\ul M$ such that, for each pair $(\bA,\bB)$ of Weil algebras,
 resp.\ for each Weil ideal $\bI$ of $\bA$,
 \begin{enumerate}
 \item
  the natural map $p^{\bA,\bB}$ is a bijection, 
  \item
the  exact  algebra sequence  $\hat \bI \to \bA \to \bA / \bI$ induces an exact bundle sequence.
\end{enumerate}
The bundle $M^{\mbA \odot \mbB} := M^{\mbA \oplus \mbB \oplus \K}$ over $M$ is called the
{\em verticle bundle with respect to $(\bA,\bB)$}.
Combining (1) and (2), we get an exact sequence over $M$
\begin{equation} \label{eqn:WeilSeq3}
\begin{matrix}
M^{\bA \odot \bB} & \to & M^{\bA \otimes_\K \bB} & \to & M^{\bA} \times_M M^\bB .
\end{matrix}
\end{equation}
 If
$\bA = \bB = T\K$, then there is is a canonical  isomorphism of Weil algebras,
$$
T\K \to T\K \odot T\K = \K \oplus (\eps_1 \K \otimes_\K \eps_2 \K) = \K \oplus \eps_1 \eps \K, \qquad 
(x + \eps v) \mapsto x +\eps_1 \eps_2 v,
$$
 whose inverse we denote by $\nu$.
It induces an isomorphism 
\begin{equation}\label{eqn:nu}
\nu : M^{T\K \odot T\K} \to TM ,
\end{equation}
between $TM$ and ``the'' vertical bundle,
and 
 we have the exact  sequence
\begin{equation}\label{eqn:tangentexact}
\begin{matrix}
M & \to & TM & \to & TTM & \to & TM \times_M TM & \to & M .
\end{matrix}
\end{equation}
\end{definition} 

\begin{remark}
From the algebra isomorphism 
$\bA \otimes (\mbB \oplus \mbB' \oplus \K) =
(\bA \otimes \mbB)  \oplus( \bA \otimes \mbB') \oplus \bA$, we get the following 
``distributivity isomorphism'' of bundles over $T^\bA M$
\begin{equation}
T^\bA (T^\bB M \times_M T^{\bB'} M) \cong T^\bA T^\bB M \times_{T^\bA M} T^\bA T^{\bB'} M  .
\end{equation}
\end{remark}

\subsection{Linear and affine  bundles induced by vector algebras and  ideals}

\begin{theorem} \label{th:WeilLinear} 
Let $\bA$ be  a vector algebra over $\K$ and  $\ul M$  a $\K$-Weil variety.
\begin{enumerate}
\item
$T^\bA M$ is a {\em linear bundle over $M$}, that is, the fibers $T^\bA_x M$ carry a natural
$\K$-module structure with origin $0_x$, and
if $\ul f : \ul M \to \ul N$ is a Weil law between  $\K$-Weil varieties, then 
tangent maps $T_x^\bA f : T^\bA_x M \to T^\bA_{f(x)} N$ are $\K$-linear. 
\item
The ``usual''
tangent bundle $TM$ and tangent maps $Tf$ are linear in fibers. 
\item
When $\ul M = \ul V$ is a $\K$-module, the linear structure is the one given by the trivialization:
in each fiber it coincides with the one coming from the
$\K$-module $V_\mbA = V \otimes_\K \mbA$, and if $\ul M$ is a Weil manifold, then
the linear structure is obtained from the linear structure underlying chart domains, for any chart.
\end{enumerate}
\end{theorem}

\begin{proof}
(1)
The structure maps
$a_M: M^\bA \times M^\bA \to M^\bA$ (fiberwise addition) and 
$ \K \times M^\bA \to M^\bA$ (fiberwise multiplication by scalars) are induced from the corresponding
canonical morphisms of vector algebras,
 $\alpha : \bA \times_\bK \bA \to \bA$, $(x,a,b) \mapsto (x,a+b)$ 
and $\rho_r : \bA \to \bA$, $(x,v) \mapsto (x,rv)$ (Lemma \ref{la:sum}): 
while $\rho_r$ induces morphisms on any Weil space, for $a_M$ we need that $\ul M$ is a Weil variety, to get 
\begin{equation}\label{eqn:add}
a= M^\alpha : M^{\bA} \times_M M^\bA  \cong M^{\bA \times_\K \bA} \to M^{\bA} .
\end{equation}
Once these structure maps are well-defined, it follows by purely ``diagrammatic'' arguments that
they define a linear bundle structure on $M^\bA$:
associativity and commutativity of $\mbA$ are expressed by commutative diagrams involving $\alpha$
and diagonal imbeddings; by functoriality, these
translate to diagrams for $a$, thus we get commutative associative products $a_x$ on the fibers $(M^\bA)_x$.
The neutral element of $a_x$ is given by the origin $0_x$: 
this is again a diagrammatic argument, following from 
$\id_\bA = \alpha \circ (\id \times_\K \zeta) = \alpha \circ (\zeta \times_\K \id)$
which holds on the level of Weil algebras.
Thus we get a monoid, and we write $a_x(u,v)=:u+v$.
Concerning inverses in this monoid,
 the map $(-1)_\bA := \rho_{-1} : \bA \to \bA$, $(a,r) \mapsto (-a,r)$ (Lemma \ref{la:sum}) is an automorphism of $\bA$ and
hence induces an Weil automorphism law $(-1)_{M^\bA}$ of $M^\bA$, which in the fiber over $x$ 
is nothing but the inversion map needed in order to get the group structure $(T_x^\bA M,+)$.
Now it follows by the same kind of arguments that the maps $\rho_r$ induce a scalar action on
$(M^\bA,a)$. 
Summing up, we have defined a $\K$-module structure in each fiber over $M^\bA$.
These definitions are natural (defined in terms of structure data of the Weil algebras), 
hence commute with natural transformations $\ul f$; and this means that $f^\bA$ is linear in fibers.

\ssk
(2) is the special case of the vector algebra $\bA = T\K =\K[X]/(X^2)$.

\ssk
(3) For $\ul M=\ul V$, the map $a$ from (\ref{eqn:add}) is nothing but vector addition in
$V_\mbA$, since $\alpha$ is addition in $\mbA$. Similarly, the scalar action is the usual one,
since it is induced from the usual scalar action of $\K$ on $\mbA$.
\end{proof}

\begin{theorem}\label{th:vectorideal}
Assume $\bI$ is a vector ideal in a Weil algebra $\bA$.
Then $M^\bA$ is an {\em affine bundle over the base $M^{\bA / \bI}$, modelled on the
linear bundle $M^{\hat \bI}$}.
More precisely, there is a natural {\em action morphism of Weil varieties}
$$
M^{\hat \bI} \times_M M^\bA \to M^\bA
$$
such that, for $x \in M$ fixed, the total space $E=(M^\bA)_x$ is a principal bundle over
the base $E/V=(M^{\bA / \bI})_x$, with structure group the additive group $V=(M_{\hat \bI})_x$
acting freely on $E$.
\end{theorem}

\begin{proof}
The proof follows exactly the lines of the one of Theorem \ref{th:WeilLinear}:
the  corresponding action morphism on the level of Weil algebras is given by Lemma
\ref{affine-lemma}; it is
compatible with the given data and hence induces a morphism on the level of varieties, which
has the required properties. 
\end{proof}

\nin 
If, with notation from the theorem,
for $u,v \in E$ belonging to the same $V$-orbit, we denote by $u-v \in V$ the unique element
$g \in V$ such that $g.v=u$,  we get a
``difference morphism''
\begin{equation}\label{eqn:difference}
M^\bA \times_{M^{\bA / \bI}} M^\bA \to  M^{\hat \bI}, \quad (u,v) \mapsto u-v .
\end{equation}
If $\ul M = \ul V$ for a $\K$-module $V$, then
$V^\bA = V_\bI \oplus V_C \oplus V$, and
$(u,y,x), (u',y',x')$ belong to the same fiber over $V_{\bA/\bI} = V_C \oplus V$ iff
$x=x'$ and $y=y'$, and then the difference $u-v$ we have defined by (\ref{eqn:difference}) coincides with the 
difference $u-v$ in the module $V_\bI$.
Applying Theorem  \ref{th:vectorideal} to the sequence (\ref{eqn:WeilSeq}) of $\bA \otimes \bB$, we get, if both 
$\bA$ and $\bB$ are vector algebras
(so the vertical ideal $\mbA \otimes \mbB$ is a vector ideal in $\bA \otimes \bB$):

\begin{corollary}
Assume $(\bA,\bB)$ is a pair of vector algebras over $\K$, and $\ul M$ a Weil variety.
Then the bundle
$M^{\bA \otimes \bB}$ is an {\em affine bundle over $M^\bA \times_M M^\bB$, modelled on
the vertical bundle $M^{\bA \odot \bB}$}.
\end{corollary}

\nin  In order to fix notation,
let us consider the  case $\bA = \K[\eps_1]= T\K$, 
$\bB= \K[\eps_2]=T\K$ and $\bA \otimes_\K \bB = \K[\eps_1,\eps_2] = TT\K$.
On the level of Weil algebras, we have the following diagrams of injections (zero sections) and
surjections (projections):
\[
\begin{matrix}
 & & TT\K & & \cr
& \buildrel {\zeta_1 \quad} \over \nearrow & & \buildrel{\quad \zeta_2} \over \nwarrow & \cr
T\K & \rightarrow & T\K \times_\K T\K & \leftarrow & T\K \, , \cr
&\nwarrow &  \uparrow & \nearrow & \cr
& & \K & & 
\end{matrix}
\qquad \quad 
\begin{matrix}
 & & TT\K & & \cr
& \buildrel {p_1 \quad} \over \swarrow & \phantom{p_{12}}\downarrow { }^{p_{12}} & \buildrel{\quad p_2} \over \searrow & \cr
T \K& \leftarrow & T\K \times_\K T\K  & \rightarrow & T\K \cr
&\searrow &  \downarrow & \swarrow & \cr
& & \K & & 
\end{matrix}
\]
Note that
there is no ``zero section $\zeta_{12}$'' (the inclusion map is not a morphism of Weil algebras; it is precisely at this point
that {\em connections} come into play, cf.\ \cite{Be08, So12}). 
On the level of second tangent bundles, this gives rise to diagrams
\[
\begin{matrix}
 & & TTM & & \cr
& \buildrel {\zeta_1 \quad} \over \nearrow & & \buildrel{\quad \zeta_2} \over \nwarrow & \cr
TM & \rightarrow & TM \times_M TM & \leftarrow & TM \, , \cr
&\nwarrow &  \uparrow & \nearrow & \cr
& & M & & 
\end{matrix}
\qquad 
\begin{matrix}
 & & TTM & & \cr
& \buildrel {p_1 \quad} \over \swarrow & \phantom{p_{12}}\downarrow { }^{p_{12}} & \buildrel{\quad p_2} \over \searrow & \cr
T M& \leftarrow & TM \times_M TM & \rightarrow & TM \cr
&\searrow &  \downarrow & \swarrow & \cr
& & M & & 
\end{matrix}
\]
where $p_{12}$ is an affine bundle with translation bundle $TM$ acting on $TTM$. 
Moreover, on each bundle $M^{\bA \otimes \bA}$, hence also on $TTM$,  there is a {\em canonical flip} 
\begin{equation}\label{eqn:flip}
\tau : TTM \to TTM
\end{equation}
 induced by the flip of $TT\K$ exchanging
$\eps_1$ and $\eps_2$ (Lemma \ref{affine-lemma}, Part (2)).

\begin{remark}\label{rk:properties}
The category of Weil varieties is {\em cartesian closed} (see Remark  \ref{rk:cclosed}), and  the 
zero set $\ul{Z(S)} $ (Section \ref{subsec:Zf}) in a Weil variety is again a Weil variety; in particular affine algebraic Weil spaces are
Weil varieties (details will  be given elsewhere).
\end{remark}

\section{Weil Lie groups}
\label{sec:WeilLie}


\subsection{Definition and examples}

\begin{definition}
A {\em Weil Lie group} is a group object in the category of Weil varieties: it is given by 
 $(\ul G, \ul m, \ul i, \ul e)$ where $\ul G$ is a Weil variety, and Weil laws
$\ul m : \ul G \times \ul G \to \ul G$ and $\ul i:\ul G \to \ul G$ and
$\ul e: \ul 0 \to \ul G$ (so $e^\bA \in G^\bA$ is a distinguished element)
 such that for each Weil algebra $\bA$,
 we have a group
$(G^\bA,m^\bA,i^\bA,e^\bA)$.

\ssk
A  {\em law of Weil Lie groups} is a natural transformation $\ul f$
between two Weil Lie groups $\ul G,\ul H$: 
thus $f^\bA:G^\bA \to H^\bA$ is a group homomorphism, for
all $\bA$.
We say that $\ul G$ is a Weil-Lie goup {\em with atlas} if it is also a Weil manifold such
that the group laws are morphisms of Weil manifolds. 
\end{definition}

\begin{example}
\begin{enumerate}
\item
 A usual (real, or complex) analytic Lie group $G$ is a Weil-Lie group, where
$G^\bA$ is the Weil bundle over $G$, which is a Lie group (\cite{KMS});
and this holds more generally for a smooth Lie group over an arbitrary topological field
(\cite{BeS}).
\item
With $\K=\Z$, $\ul{\GL(n,\Z)}$, $\ul{\SL(n,\Z)}$,$\ul{\OO(n,\Z)}$, etc., 
are the Weil Lie groups over $\Z$ assigning to a
$\Z$-Weil algebra $\bA$ the groups $\GL(n,\bA)$, $\SL(n,\bA)$, resp.\ $\OO(n,\bA)$.  Note that
$\GL(n,\Z)$ has a natural atlas (modelled on $M(n,n;\Z)$), and similarly for $\OO(n,\Z)$ (having the 
{\em Cayley rational chart},
modelled on ${\rm Skew}(n,\Z)$),
but for $\SL(n,\Z)$ it is more difficult to define an atlas modelled on the space of integer matrices of trace zero.
\item
%
For any $\K$-module $V$,  $\Gl_\K(V)$ is a Lie group (modelled on the module
$\End_\K(V)$). More generally, for any unital associative $\K$-algebra ${\mathcal A}$, the functor of
invertible elements ${\mathcal A}^\times$ defines a Weil Lie group, modelled on $\mathcal A$.
\item
The Weil laws from a Weil space $\ul M$ into some fixed Weil Lie group $\ul G$ form
again a Weil Lie group with ``pointwise'' multiplication  (``loop group'').
\item
Let $\g$ be any algebra over $\K$ (Lie, Jordan, or other). Then the $\K$-algebra automorphisms give rise to a Weil variety (associating $\Aut_\bA(\g^\bA)$ to $\bA$), and
hence to  a Weil Lie group $\ul \Aut_\K(\g)$. 
In finite dimension over $\R$, there is an atlas given by the exponential map, but in general 
there is no such atlas.
\item
Let $\K = {\mathbb Q}$ 
and $\g$ be a Lie algebra over $\K$. For each Weil algebra $\bA = \K \oplus \mbA$
let $G^\bA := \g \otimes_\K \mbA$ be equipped with its Baker-Campbell-Hausdorff multiplication:
$X \cdot_\bA Y = X+Y + \frac{[X,Y]}{2} + ...$ (finite sum since $\mbA$ is nilpotent).
Then $\ul G$ is a Weil Lie group (this can be seen as an infinitesimal version of Lie's Third Theorem).
Note that here $G = G_\K$ is just a point.
\end{enumerate}
\end{example}

\subsection{First approximation: vector addition}
If $\ul G$ is a $\K$-Weil Lie group, then for each Weil algebra $\bA$,
$(\ul{G^\bA},\ul{m^\bA})$ is a Weil Lie group (defined over $\bA$);
it associates to a Weil algebra $\bB$ the group $G^{\bA \otimes \bB}$.
For instance, the
{\em tangent group} $\ul{TG} = \ul G^{T\K}$ is again a Weil Lie group, and there is an
exact sequence of Weil Lie groups
\begin{equation}\label{eqn:LieSequence}
\begin{matrix}
0 & \to & \ul{T_e G} & \to & \ul{TG} & \to & \ul G & \to & 1 
\end{matrix}
\end{equation}
where $\ul{T_e G}$ is the functor associating to a Weil algebra $\bA$ the group
$(G^{T\bA})_e$, fiber of $G_{T \bA}$ over $e \in G$. 
This sequence is split with section $\ul \zeta:\ul G \to \ul{TG}$ induced by
the inclusion $\K \subset T\K$. 

\begin{theorem}\label{th:Lie1}
The kernel $\ul \g := \ul{T_e G}$ appearing in the sequence (\ref{eqn:LieSequence}) 
is an abelian Weil Lie group, and its group law coincides with the additive law of the
tangent space of $\ul G$ at $\ul e$.
The group $G$ acts on by conjugation on $\g$ is via linear maps; this representation 
is called the {\em adjoint representation}
$$
\Ad : G \to \Gl_\K (\g) , \quad g \mapsto \Ad(g) = (X \mapsto gXg\inv) .
$$
The group $TG$ is a semidirect product $G \ltimes \g$ via $\Ad$.
\end{theorem}

\begin{proof}
The first differential $T_{(e,e)} m$ is linear,  so the group law is, for
$X,Y \in \g$,
$$
T_{(e,e)} m (X,Y) =
T_{(e,e)} (X,0_e) + T_{(e,e)} (0_e,Y) = X + Y  ,
$$
where $0_e$ is the neutral element of $TG$. Since $\Ad(g)$ is the tangent map of
$G \to G$, $x \mapsto gxg\inv$ at $e$, it is linear. 
Since the sequence (\ref{eqn:LieSequence}) is split exact, $TG$ is a semidirect product.
\end{proof}


\begin{example}\label{ex:tangentgroup}
We compute the tangent group of $\ul G = \ul{\GL_\K(V)}$:
by definition, $TG = \GL_{T\K} (TV)$, which is the same as
$\GL_\eps (V \oplus \eps V)$, the invertible $\K$-linear maps $F$ on $V \oplus \eps V$
commuting with $\eps = \bigl(\begin{smallmatrix} 0 & 0 \cr 1 & 0 \end{smallmatrix}\bigr)$.
The condition
$\bigl(\begin{smallmatrix} 0 & 0 \cr 1 & 0 \end{smallmatrix}\bigr)
\bigl(\begin{smallmatrix} a & b \cr c & d \end{smallmatrix}\bigr) =
\bigl(\begin{smallmatrix} a & b \cr c & d \end{smallmatrix}\bigr) 
\bigl(\begin{smallmatrix} 0 & 0 \cr 1 & 0 \end{smallmatrix}\bigr)$
leads to $a=d$, $b=0$, whence $F = F_{a,c}=\bigl(\begin{smallmatrix} a & 0 \cr c & a \end{smallmatrix}\bigr): (x+\eps v) \mapsto  a x + \eps (av + cx)$:
\begin{equation*}
TG = \{ F_{a,c}  \in \End_\K (TV) \mid \, 
a \in \GL_\K (V), c \in \End_\K (V) \} 
\end{equation*}
with product $F_{a,c} \circ F_{a',c'} = F_{aa',ac'+a'c}$.
\end{example}

\begin{example}
Let $\beta: V \times V \to V$ be a $\K$-bilinear map. We compute the tangent group of 
the automorphism group  $\ul G :=\ul{\Aut_\K(V,\beta)}$:
$TG$ is the group of all $F_{a,c}$ that preserve
$T\beta(x+\eps v, x' + \eps v') = \beta(x,x') + \eps (\beta(x,v') + \beta(x,v'))$.
This implies that $a \in \Aut_\K(V,\beta)$, and for $a=\id$, we get the condition
that $c$ is a derivation of $\beta$. Thus $TG$ is a semidirect product of
$\Aut_\K(V,\beta)$ with the vector group
\begin{equation*}
\Der_\K(V,\beta) = \{ c \in \End_\K(V) \mid \forall x,x' \in V :  c \beta(x,x')=\beta(cx,x')+\beta(x,cx') \} .
\end{equation*}
\end{example}

\subsection{Second approximation: Lie bracket}
As is well-known, the Lie bracket associated to a Lie group is defined in terms of
{\em second derivatives} of the group structure.  In order to explain this in a conceptual
and functorial way, consider  the sequence (\ref{eqn:tangentexact}) and
the diagrams for $TTM$ given at the end of the preceding 
section, and replace $M$ by $G$: then all spaces in questions are groups and all canonical maps
are group morphisms.  The group  $(TTG)_e$ is in general non-abelian, and the Lie bracket will be
defined in terms of  the {\em group commutator law of $G$}:
\begin{equation}
\ul\Gamma : \ul G \times \ul G  \to \ul G , \qquad 
\Gamma^\bA (g,h)=ghg\inv h\inv, \, \forall g,h \in G^\bA .
\end{equation}

\begin{theorem}[Lie bracket]\label{th:Liebracket}
Let $\ul G$ be a Weil Lie group over $\K$. Then, with $\nu$ defined by Eqn.\ (\ref{eqn:nu}),  the law
$$
\ul \Lambda := \ul \nu  \circ \ul \Gamma \circ (\ul \zeta_1 \times \ul \zeta_2) :
\ul \g \times \ul \g \to \ul \g
$$
defines a Lie algebra law on $\ul \g$. On the level of the second tangent group $TTG$, this
formula means that the Lie bracket is given for $X,Y \in \g$ by
$$
[ X,Y] := \Lambda^\K (X,Y) = \nu  ( \eps_1 X \cdot \eps_2 Y (\eps_1 X)\inv \cdot (\eps_2 Y)\inv ) .
$$
This formula can be rewritten, equivalently, by
using the notation $[ \, , \, ]_{grp}:=\Gamma$ for the group commutator, 
and by using a minus sign to denote the  difference of two elements of $(TTG)_e$  lying in the same fiber 
over $(TG \times_G TG)_e = \g \times \g$:
\begin{align*}
[\eps_1 X,\eps_2 Y]_{grp}  &=  \eps_1 \eps_2 [X,Y]_{alg}   \cr
\eps_1 X \cdot \eps_2 Y \cdot (\eps_1 X)\inv  - \eps_2 Y & =   \eps_1 \eps_2 [X,Y]_{alg}  \cr
\eps_1 X \cdot \eps_2 Y -  \eps_2 Y \cdot \eps_1 X  &=   \eps_1 \eps_2 [X,Y]_{alg} 
\end{align*}
\end{theorem}

\begin{proof} 
In the following proof, we write  arguments using elements on the level of $TTG$;
it is possible to formulate things in an element-free way using diagrams (hence 
showing that actually all identities are laws over $\K$); but this would make the proof
longer and harder to read, and so we leave this task to the reader.

\ssk
First of all, the maps are well-defined, and all three  formulae for the Lie bracket agree:
let $X,Y \in \g$; since $p_{12}$ defines a group morphism $(TTG)_e \to \g \times \g$,
and $\g \times \g$ is abelian, we have
$$
p_{12}(\eps_1 X \cdot \eps_2 Y ) = p_{12}(\eps_1 X)  + p_{12}(\eps_2 Y) = \eps_1 X + \eps_2 Y=
p_{12} (\eps_2 Y \cdot \eps_1 X)
$$ 
and likewise
$p_{12} (\eps_1 X \cdot \eps_2 Y \cdot (\eps_1 X)\inv  \cdot (\eps_2 Y)\inv  ) = 0$.
It follows that $Z:=[\eps_1X,\eps_2 Y]_{grp} $ belongs to the vertical fiber over $e$, whence 
$\nu  [\eps_1X,\eps_2Y]_{grp} \in T_e G=\g$ is well-defined.
Moreover,
$Z \cdot   \eps_2 Y \eps_1 X = \eps_1X \eps_2 Y$, whence
$Z = \eps_1X \eps_2 Y - \eps_2 Y \eps_1 X $ by definition of the affine bundle
$TTG$ over $TG \times_G TG$ in the preceding chapter.
Similarly, $Z$ can be written as
$Z = \eps_1 X \cdot \eps_2 Y \cdot (\eps_1 X)\inv  - \eps_2 Y$.

\ssk
Let us show that $[Y,X]=-[Y,X]$.
Classically, this relies on symmetry of second differentials; in our setting, this corresponds to
the fact that the preceding construction is natural and hence commutes with the 
flip $\tau:TTG \to TTG$ exchanging
$\eps_1$ and $\eps_2$ (equation (\ref{eqn:flip}),
that is,
\begin{align*}
\eps_1 \eps_2 [X,Y] & =
\tau (\eps_1 \eps_2 [X,Y]) =
\tau ( \eps_1 X \cdot \eps_2 Y \cdot (\eps_1 X)\inv \cdot (\eps_2 Y)\inv ) \cr
&=
\eps_2 X \eps_1 Y ( \eps_2 X)\inv \cdot (\eps_1 Y)\inv  = 
(\eps_1 \eps_2 [Y,X])\inv .
\end{align*}
Now, inversion in the vertical bundle  is  $Z \mapsto -Z$, whence
$[X,Y]= - [Y,X]$.

\ssk
In order to show that the bracket is $\K$-bilinear, we use the following fact, which is
obvious in differential calculus, and whose proof in the category of Weil spaces is
is rather straightforward and will be omitted here:
for a binary map $\ul f: \ul M \times \ul N \to \ul P$ and $u\in TTM$, $v \in TTN$, we have
maps $Tf(u,\cdot):TN \to TP$ and $Tf(\cdot,v) :TM \to TP$, and then
$$
TTf (u,v) = T (Tf(\cdot,v))(u) = T(Tf(u,\cdot))v .
$$
Apply this to the commutator map $f=\Gamma$ and $u=\eps_1 X$, $v=\eps_2 Y$ to
see that the expression is linear in $X$ and in $Y$ (since tangent maps for Weil varieties
are linear). 

\ssk
Finally, in order to prove the Jacobi identity, various strategies are possible.
The simplest is  probably the one given in \cite{DG}, II.\S 4, no. 4,
following the classical fact that $[X,Y]=\ad(X)Y$, where $\ad$ is the derivative of the adjoint
representation $\Ad$ at $e$: 
the construction of the Lie bracket is natural, that is, $G$ acts via the adjoint representation
by {\em automorphisms} of the Lie algebra, so  we have a morphism
$$
\ul \Ad : \ul G \to \ul{\Aut_\K(\g, [ \, , \, ])} .
$$
In particular, $T\Ad$ is a morphism $TG \to T(\Aut(\g))$, and hence also its restriction to
fibers over neutral elements is a morphism ($\K$-linear map, since these fibers are just
vector groups). From Exemple \ref{ex:tangentgroup}  we know that the fiber of $T(\Aut(\g))$ over $\id$ is
$\Der(\g)$, hence we have
$$
\ad := T\Ad \vert_\g : \g \to \Der(\g) .
$$
The formula $\eps_1 X \cdot \eps_2 Y \cdot (\eps_1 X)\inv  =  \eps_2 Y  +   \eps_1 \eps_2 [X,Y]_{alg}$
already established above shows that
 $[X,Y]=\ad(X)Y$, and saying that $\ad(X) \in \Der(\g)$
gives the Jacobi identity.
\end{proof}

\begin{remark}  Some words on other proofs of the Jacobi identity.
In one way or another, one has to use third derivatives.
A proof that can be adapted to the present context is given in \cite{Be08}, p.\ 117:
the higher order tangent groups $TG, TTG, T^3 G$ can be trivialized from the left or
from the right, and the group law can then be described by explicit formulae. 
Using this, expanding  $\eps_3 Z \cdot (\eps_2 Y \cdot \eps_1 Z)
= (\eps_3 Z \cdot \eps_2 Y) \cdot \eps_1 Z$ in two different ways and
comparing, the Jacobi identity drops out. 
This proof is closely related to the proof by using Hall's identity (see \cite{Ko10}, p.\ 219  for a detailed 
discussion, following Serre's proof).   
Also, the recent work \cite{Viz} is closely related to this --
it would be very interesting, if the trivialization of the jet bundles $J^k G$ used in
loc.\ cit.\  can be defined in our categorical context; then this would lead to another 
approach to the Lie bracket in terms of the jet bundles $J^2 G$, $J^3 G$.
Finally, most textbooks define the Lie algebra of usual Lie groups
$G$ in terms of left invariant vector fields, see next section.
\end{remark}

\subsection{Weil Lie torsors (grouds) and symmetric spaces}\label{sec:ss} 

\begin{definition}
A {\em Weil variety with binary (ternary) multiplication} is a Weil variety  $\ul M$ together with a  binary or ternary Weil law
$\ul \mu: \ul M \times \ul M \to \ul M$, resp.\ $\ul \mu : \ul M \times \ul M \times \ul M \to \ul M$.
{\em Morphisms} are Weil laws compatible with the respective  multiplication laws. 
\end{definition}

\begin{definition}
A {\em Weil Lie torsor}\footnote{Besides ``torsor"", other existing terminology is: {\em groud, heap, pregroup,...}, cf.\  \cite{BeKi1}.}
 is a Weil variety $\ul G$ together with a  Weil law
$\ul \mu : \ul G^3 \to \ul G$, such that, for each Weil algebra $\bA$, writing
$(xyz)^\bA$ instead of $\mu^\bA(x,y,z)$,

\begin{enumerate}
\item
$\forall x,y \in G^\bA$: $(x x y)^\bA = y = (yxx)^\bA$,
\item
$\forall x,y,u,v,w \in G^\bA$: $(xy(uvw)^\bA)^\bA) = ((xyu)^\bA vw)^\bA) = (x(vuy)^\bA w)^\bA$.
\end{enumerate}
\end{definition}

Every Weil Lie group $(\ul G,\ul m)$ with the law
$(xyz)^\bA := m ^\bA (x,m^\bA (y,z))$ is a Weil Lie torsor, and, conversely, any choice of base point $\ul y$ in
a Weil Lie torsor gives a Lie group with neutral element $\ul y$ and product
$m^\bA (x,y^\bA,z)$ (cf.\ \cite{BeKi1}). 
Thus Lie theory of Weil Lie torsors is the same as ``base point-free Lie theory''.
For instance, rewriting Theorem \ref{th:Liebracket} in a base-point free way gives a {\em field of Lie brackets},
which can also be interpreted as the {\em torsion tensor of a canonical connection} (cf.\ \cite{Be08}):

\begin{theorem}
On a Weil Lie torsor $\ul G$, consider the Weil law given for each $\bA$ by
$$
\Gamma^\bA (x,y,z):= ( x(xyz)^\bA z)^\bA = ((xzy)^\bA xz)^\bA = (xz (yxz)^\bA )^\bA .
$$
Then the law
$$
\ul \Lambda := \ul \nu \circ \ul \Gamma \circ (\zeta_1 \times \zeta_2) : \ul{TG} \times_G \ul{TG} \to \ul{TG}
$$
defines a field of Lie algebra laws on $\ul G$. It can also be written as
$$
\bigl(\eps_1 X , 0_y , \eps_2 Z \bigr) - 
\bigl(\eps_2 Z, 0_y, \eps_1 X \bigr) =
\eps_1 \eps_2 [X,Z]_y.
$$
\end{theorem}

\begin{definition}
 A  {\em (Weil) symmetric space} is a Weil space with a  binary Weil law
$\ul s: \ul M \times \ul M \to \ul M$ such that, for each Weil algebra $\bA$,

\begin{enumerate}
\item
$\forall x \in M^\bA$: $s^\bA(x,x)=x$
\item
$\forall x,y \in M^\bA$: $s^\bA (x,s^\bA(x,y)) = y$
\item
$\forall x,y,z \in M$: $s^\bA (x , s^\bA (y , s^\bA(x,z))) = s^\bA( s^\bA(x,y),z)$
\item
$\forall x \in M^\bA$, $y \in T_x (M^\bA)$: $s^{T\bA}(0_x,y) = -y$.
\end{enumerate}
\end{definition}

\begin{theorem}\label{th:LG-SS}

\begin{enumerate}
\item
For each Weil algebra $\bA$,
the Weil bundle $\ul{M}^\bA$ of a symmetric space $\ul M$ is again a symmetric space.
\item
Every Weil Lie group $\ul G$ becomes a symmetric space when equipped with the law
$s^\bA(x,y)=xy\inv x$.
\end{enumerate}
\end{theorem}

\begin{proof}
(1)  All four
 axioms can be expressed by commutative diagrams
involving  the structure map $s$ and natural maps such as the diagonal imbedding, so that 
applying a functor $T^\bB$ yields structures of the same type. 
The diagrams for the first three axioms are given in \cite{Lo69}, p.\ 75;  to write axiom (4) as a diagram,
use zero section $\zeta:M^\bA \to TM^\bA$ and fiberwise multiplication by $-1$,
$(-1)_{TM^\bA}:TM^\bA \to TM^\bA$.

(2)  Any abstract group $G$ with $s(x,y)=xy\inv x$ satisfies (1), (2), (3). If $\ul G$ is a Weil Lie group, then 
 (4) holds since, in a Lie group,
 the tangent map of inversion at the origin is $-\id_{T_e G}$.
\end{proof}


\nin
Now one can define a {\em Lie triple system} attached to a symmetric space, see
 Cor.\  \ref{cor:LTS}.

\section{Vector fields  and groups of infinitesimal automorphisms}

\subsection{Structure of $\bA$-automorphism groups}
The following results provide  important general tools for differential geometry.
Since the proofs follow closely the ones given
\cite{Be08}, Chapter 28, our account  will be rather concise.

\begin{definition}
Let $\ul M$ be a Weil space and $\bA$ a Weil algebra.
An {\em infinitesimal endomorphism} of the Weil bundle $\ul M^\bA$ is an $\bA$-endomorphism
$$
\ul F : \ul{M}^\bA \to \ul{M}^{\bA}
$$
preserving fibers over $\ul M$, that is, such that $\ul p^\bA \circ \ul F = \ul p^\bA$.
It is called an {\em infinitesimal automorphism} if, moreover, it is an automorphism of $\ul{M}^\bA$.
Obviously, the infinitesimal automorphisms form a group, which we denote by
$\Infaut_\bA(\ul M^\bA)$.
\end{definition}

\begin{theorem}
\label{th:Aut1} 
For any Weil space $\ul M$,
the group $\Infaut_\bA(\ul M^\bA)$ is a normal subgroup of $\Aut_\bA(\ul M^\bA)$; more precisely,
it fits into the following exact sequence of groups
$$
\begin{matrix} 
1 & \to & \Infaut_\bA(\ul M^\bA) & \to & \Aut_\bA(\ul M^\bA) & \to & \Aut_\K(\ul M) & \to & 1 
\end{matrix}
$$
which splits via $\Aut_\K(\ul M) \to \Aut_\bA(\ul M^\bA)$, $\ul g \mapsto \ul g^\bA$.
\end{theorem}

\begin{proof}
The Weil algebra projection $\bA \to \K$ induces, for any $\bA$-Weil law $\ul F : \ul M^\bA \to \ul N^\bA$,
a map $F^\K : (M^\bA)^\K \to (N^\bA)^\K$, and similarly a law $\ul F^\K$(recall Theorem \ref{th:transitivity}).
Now define the $\K$-Weil law 
$\ul f := \ul p \circ \ul F^\K \circ \ul z:\ul M \to \ul N$ (where $p:(N^\bA)^\K \to N$ is projection and
$z:M \to (M^\bA)^\K$ injection).
Let us call $\ul f$ the {\em base law of $\ul F$}.
It depends functorially on $\ul F$; in particular, for $\ul M = \ul N$, it defines a group morphism
$\Aut_\bA(M^\bA)  \to  \Aut_\K(\ul M)$ whose kernel is formed by automorphisms having as base law the identity
law on $\ul M$, that is,  the infinitesimal automorphisms.
It is also clear that for $\ul g:\ul M \to \ul M$, the base law of $\ul g^\bA$ is $\ul g$ itself, hence the
group  morphism splits via $\ul g \mapsto \ul g^\bA$, and thus the sequence is exact.
\end{proof}

\begin{definition}
Let $\ul M$ be a Weil space and $\bA$ a Weil algebra.
An {\em $\bA$-vector field} is a section of the canonical projection
$\ul p^\bA : \ul M^{\bA} \to \ul M$, that is, a $\K$-Weil law
$
\ul X : \ul M \to \ul M^{\bA}
$
such that $\ul p \circ \ul X = \ul\id_M$.
%
\end{definition}

\begin{theorem}
\label{th:Aut2}
Let $\ul M$ be a Weil space and $\bA$ a Weil algebra.
\begin{enumerate}
\item
There is a canonical bijection between $\bA$-vector fields $\ul X$ and infinitesimal endomorphisms $\ul F$, 
given by associating to $\ul F$ the $\bA$-vector field $\ul X := \ul F \circ \ul \zeta$, and to $\ul X$ the infinitesimal
endomorphism given by 
$$
\ul F:= \ul \mu \circ \ul X^\bA : \ul M^\bA \to  \ul M^{\bA \otimes \bA} \to \ul M^\bA,
$$
where $\mu  : \ul M^{\bA \otimes \bA} \to \ul M^\bA$ is the law induced by
$\mu : \bA \otimes \bA \to \bA$ (lemma \ref{la:sum}).
\item
The space of  $\bA$-vector fields forms a monoid  with respect to the law 
$\ul X \cdot \ul Y := \ul \mu \circ \ul X^\bA \circ \ul Y$ and neutral element the zero section. 
If $\ul M$ is a Weil variety and
  $\bA = TM$, this monoid  is the abelian group of ``usual'' vector fields with group law given by pointwise
addition in tangent spaces. 
\item
If $\ul M$ is a Weil manifold, then every infinitesimal endomorphism is an automorphism, and hence the monoid
from the preceding item is a group, for all Weil algebras $\bA$.
\end{enumerate}
\end{theorem}

\begin{proof}
(1), (2)
The proof  from \cite{So12}, Thm.\  8.2.2
 carries over almost word by word, so we will not repeat the details.
See  \cite{Be08}, Thm.\  28.1  for the case $\bA = T^k \K$.

(3)  
Both the proofs from \cite{Be08} and \cite{So12}, loc.cit., use chart arguments, and they carry over for general Weil
manifolds (but we do not know if the claim holds for general Weil varieties or even Weil spaces). 
\end{proof}

\begin{corollary}[Lie bracket of vector fields]
For any Weil variety  $\ul M$, the law $G^\bA := \Aut_\bA ( M^\bA)$ defines a Weil Lie group 
$\ul G = \ul{\Aut_\K (M)}$. Its Lie algebra is the space
$\ul \g = \ul{\XX}(TM)$ of sections of $\ul{TM}$ ( usual vector fields, identified with infinitesimal automorphisms), 
with Lie bracket being  the bracket of vector fields, described by the
formula from \cite{Be08}, Theorem 14.4:
\begin{equation}
[ {\eps_1 X} , {\eps_2 Y}]_{group} = \eps_1 \eps_2 [X,Y]_{alg} . 
\end{equation}
Moreover,  if $\ul M$ is a Weil manifold,
 the bracket may be computed in a chart by the usual chart formula in terms of ordinary differentials
 (as defined in equation (\ref{eqn:Adiff})):
\begin{equation}
[X,Y](x)=dX(x) Y(x) - dY(x)X(x) .
\end{equation}
\end{corollary}

\begin{proof}
It is straightforward that $\bA \mapsto G^\bA$ defines a group object in the category of Weil spaces, but it is less
obvious that this object is again a Weil variety. Here we use the description of the fibers given by the preceding theorem:
since a group is homogeneous, if suffices to consider the fiber over 
 the neutral element $\ul{\id_M}$; by the preceding theorem, this fiber
  is given by the collection of $\bA$-vector fields over $M$, and using this,
it is quite easy to deduce that the defining properties of a Weil variety hold for $\ul G$.  
Now the Lie algebra of $\ul G$ is the fiber of $\Aut_{T\K} (TM)$ over $\id_M$, that is, the space of usual vector fields,
and the construction of the Lie bracket from the preceding chapter translates to the context of vector fields as
described in the theorem.
\end{proof}

\begin{remark}
Conversely, the Lie algebra of a Lie group can also be described as the space of left (or right) invariant vector fields
on the group, with bracket given by the bracket of vector fields.
\end{remark}

\begin{remark}\label{rk:cc}\label{rk:cclosed}
The argument from the preceding proof can be extended  to show that the category of Weil varieties
has internal hom spaces and hence is cartesian closed.
\end{remark}

\subsection{Automorphisms and derivations}

\begin{theorem}
Let $\ul M$ be a Weil variety with binary  multiplication $\ul s: \ul M\times \ul M \to\ul M$.
For  a vector field $\ul X: \ul M \to \ul{TM}$, the following are equivalent:

\begin{enumerate}
\item
$\ul X$ is a homomorphism of multiplications $\ul s$ and $\ul{Ts}$,
\item
the infinitesimal automorphism $\ul F$ corresponding to ${\ul X}$ is an automorphism of multiplication $\ul{Ts}$,
\item
the vector field $\ul X$ is a {\em derivation}: for each Weil algebra $\bA$ and $p,q \in M^\bA$,
$$
X^\bA ( s^\bA(p,q)) = Ts^\bA (X^\bA(p),0_q) + Ts^\bA(0_p,X^\bA(q)),
$$
or, in notation following \cite{Lo69}, Lemma 4.2, 
$$
X (p \cdot q) = X(p) \cdot q + p \cdot X(q) .
$$
\end{enumerate}

\nin The {\em derivations of $s$} (vector fields satisfying (3)) form a Lie subalgebra 
$\ul{\Der(M,s)}$ of the Lie algebra of
vector fields. 
A similar result holds for general $n$-ary Weil laws on a Weil variety.
\end{theorem}

\begin{proof} 
The proof  from \cite{Lo69}, Lemma 4.2 and Prop.\ 4.3 a),  carries  over in a rather 
straightforward way.
\end{proof}

\begin{corollary}\label{cor:LTS}
The derivations of a symmetric space $(\ul M,\ul s)$ form a Lie subalgebra of the Lie algebra of vector fields.
If $2$ is invertible in $\K$, then, for any choice of base point $\ul o$ in $\ul M$, this Lie algebra
inherits a $\Z / 2 \Z$-grading induced from the symmetry with respect to $\ul o$, and the $-1$-eigenspace
is in canonical bijection with the tangent space $T_o M$; in this way, the
{\em Lie triple system} of a symmetric space attached to $\ul o$ can be defined as in
\cite{Be08, Lo69}. It depends functorially on the symmetric space with base point.
\end{corollary}

\begin{remark}
The Lie triple system can be interpreted as {\em curvature tensor of a canonical connection},
cf.\ \cite{Be08, Lo69}.
\end{remark}

\begin{remark}
If $\ul M$ is a Weil variety with multiplication $\ul s$, then
then the automorphisms of $\ul s$ form a Weil Lie group $\ul\Aut(\ul M,\ul s)$, 
Lie subgroup of the group $\ul\Aut_\K(\ul M)$,
given by associating to
a Weil algebra $\bA$ the group
$$
\Aut_\bA(M^\bA,s^\bA) := \{ g^ \bA \mid \, \ul g \in \ul{\Aut(M)}, \,
\ul s \circ (\ul g \times \ul g) = \ul g \circ \ul s \} ,
$$
and $\ul \Der(\ul M,\ul s)$ is its Lie algebra. The main point here is to show that $\ul\Aut(\ul M,\ul s)$
is a Weil variety,  cf.\ remark \ref{rk:properties}  on categorical properties  of 
Weil varieties.  In this context, the higher order analog of the derivation property 
is  the {\em expansion property} (see \cite{KMS}, 37.6).
This will be discussed in more detail elsewhere.
\end{remark}

\end{document}